\documentclass[english,12pt]{article}

\pdfoutput=1

\usepackage{fullpage}
\usepackage{graphicx}
\usepackage{amsmath}
\usepackage{amsrefs}
\usepackage{amsthm}
\usepackage{amssymb}
\usepackage{array}
\usepackage{enumerate}
\usepackage{multicol}
\usepackage{tikz}
\usepackage{tikz-qtree}
\usepackage{hyperref}

\newtheorem{theorem}{Theorem}
\newtheorem{lemma}{Lemma}

\newcommand{\p}[1]{(#1)}
\newcommand{\pb}[1]{\left(#1\right)}

\newcommand{\st}[1]{\left\{#1\right\}}

\newcommand{\abk}[1]{\left<#1\right>}
\newcommand{\fl}[1]{\left\lfloor#1\right\rfloor}

\newcommand{\quot}[1]{``#1''}

\newcommand{\mbb}{\mathbf}

\newcommand{\As}{\lambda}
\newcommand{\Bs}{\mu}
\newcommand{\Cs}{\gamma}

\newcommand{\Rh}{B}

\begin{document}
\title{The Behavior of a Three-Term Hofstadter-Like Recurrence with Linear Initial Conditions}
\author{Nathan Fox\footnote{Department of Quantitative Sciences, Canisius University, Buffalo, New York,
\texttt{fox42@canisius.edu}
}}
\date{}

\maketitle

\begin{abstract}
In this paper, we study the three-term nested recurrence relation 	$\Rh(n)=\Rh\p{n-\Rh\p{n-1}}+\Rh\p{n-\Rh\p{n-2}}+\Rh\p{n-\Rh\p{n-3}}$ subject to initial conditions where the first $N$ terms are the integers $1$ through $N$. This recurrence is the three-term analog of Hofstadter's famous $Q$-recurrence $Q(n)=Q(n-Q(n-1))+Q(n-Q(n-2))$. Nested recurrences are highly sensitive to their initial conditions. Some initial conditions lead to finite sequences, others lead to predictable sequences, and yet others lead to sequences that appear to be chaotic and infinite.
A corresponding study to this one was previously carried out on the $Q$-recurrence. As with that work, we consider two families of sequences, one where terms with nonpositive indices are undefined and a second where terms with nonpositive indices are defined to be zero. We find similar results here as with the $Q$-recurrence, as we can completely characterize the sequences for sufficiently large $N$. The results here are, in a sense, simpler, as our sequences are all finite for sufficiently large $N$.
\end{abstract}

\section{Introduction}
The Hofstadter $Q$-recurrence~\cite{geb}
\[
Q\p{n}=Q\p{n-Q\p{n-1}}+Q\p{n-Q\p{n-2}}
\]
has been the subject of numerous studies. Hofstadter's original $Q$-sequence starts with initial conditions $Q\p{1}=Q\p{2}=1$. This sequence has tantalizing properties that have thus far evaded proof, though they have been the subject of statistical studies~\cites{pinn,spotbased}. As such, most studies of the $Q$-recurrence analyze sequences generated by other initial conditions~\cites{golomb, rusk,symbhof, genrusk, gengol}. One recent approach involves studying a family of initial conditions described by a parameter~\cite{hof1thruN}.

In this paper, we apply the parametrized initial condition approach to a different recurrence, the three-term Hofstadter-like recurrence
\[
\Rh\p{n}=\Rh\p{n-\Rh\p{n-1}}+\Rh\p{n-\Rh\p{n-2}}+\Rh\p{n-\Rh\p{n-3}}.
\]
This recurrence is known to generate a well-behaved sequence~\cite{slowtrihof} when given initial conditions $\Rh\p{1}=1$, $\Rh\p{2}=2$, $\Rh\p{3}=3$, $\Rh\p{4}=4$, and $\Rh\p{5}=5$. For convenience, we refer to this well-behaved sequence as the \emph{$\Rh$-sequence}. Aside from that one article, this recurrence has not been widely reported on. This is presumably because most natural initial conditions lead to finite sequences. If a sequence $\Rh^*\p{n}$ generated by the $\Rh$-recurrence ever has $\Rh^*\p{n-1}\geq n$ or $\Rh^*\p{n-1}\leq0$, then $\Rh^*\p{n}$ would be undefined. When this sort of behavior occurs, we say that the sequence \emph{dies} after $n-1$ terms, or that it dies at index $n$. This assumes, as is standard, that the first term defined by the initial conditions is $\Rh^*(1)$. In this paper, we consider such initial conditions, but we also consider infinite initial conditions that define values for $\Rh^*(n)$ when $n\leq0$. In this realm, $\Rh^*(n-1)\geq n$ does not lead to sequence death, but $\Rh^*(n-1)\leq 0$ still does. So, to avoid some confusion later in the paper, when a sequence dies because $\Rh^*\p{n-1}\leq0$, we say that the sequence \emph{ends}.

\subsection{Notation}

Going forward, the only recurrence relation we discuss is the $\Rh$-recurrence, but we study it with many different initial conditions.  We introduce analogous notation to~\cite{hof1thruN}. The notation $\Rh\p{n}$ refers to the $n$th term of the $\Rh$-sequence itself.  The notation $\Rh^*\p{n}$ refers to a generic sequence that satisfies the $\Rh$-recurrence.  For any other specific sequence satisfying the $\Rh$-recurrence, we use $\Rh$ with a subscript that we define for that particular sequence.

We use angle brackets to denote our initial conditions.  For example, $\abk{1,2,3,4,5}$ is shorthand for the initial conditions for the $\Rh$-sequence.  Sometimes, we wish to define $\Rh^*\p{n}=0$ for all $n\leq0$ in order to prevent our sequences from dying too soon.  This convention is noted with a symbol $\bar{0}$ followed by a semicolon at the start of the initial conditions.  For example, $\abk{\bar{0};1,1,1}$ is shorthand for $\Rh^*\p{n}=0$ for $n\leq0$, $\Rh^*\p{1}=1$, $\Rh^*\p{2}=1$, and $\Rh^*\p{3}=1$.

\subsection{Structure of this Paper}
This paper's structure mirrors that of the analogous study on the $Q$-recurrence~\cite{hof1thruN}. In Section~\ref{s:wd}, we characterize the sequences generated by the $\Rh$-recurrence via family of initial conditions of the form $\abk{1,2,3,\ldots,N}$. Then, in Section~\ref{s:sd}, we study the more general initial conditions $\abk{\bar{0};1,2,3,\ldots,N}$.  Finally, we suggest some future research directions in Section~\ref{s:future}.

\section{Behavior of the $\Rh$ Recurrence with Linear Initial Conditions}\label{s:wd}

In this section, we consider sequences obtained from the $\Rh$-recurrence and initial conditions of the form $\abk{1,2,3,\ldots,N}$ for some integer $N\geq3$.  Henceforth, we denote this sequence for a given value of $N$ by $\Rh_N$.

We have the following result, which characterizes the behaviors of almost all of these sequences.

\begin{theorem}\label{thm:1thruN}
For $N=3$, $N=4$, or $N\geq10$, the sequence $\Rh_N$ dies.  Furthermore, if $N\geq14$, the sequence has exactly $N+24$ terms.
\end{theorem}

\begin{proof}
Computing terms, we obtain that $\Rh_3\p{4}=6$, $\Rh_4\p{5}=6$, $\Rh_{10}\p{1015}=1036$, $\Rh_{11}\p{117}=120$, $\Rh_{12}\p{45}=47$, and $B_{13}\p{73}=82$. So, these sequences all die~\cite{oeis}*{A373230, A373231, A373232, A373233}.

Now, we treat $N$ as a symbolic parameter and apply the process from~\cite{hof1thruN}. This allows us to compute $24$ terms following the initial conditions.  These $24$ terms are:
\begin{align*}
6,\, &N + 1,\, N + 2,\, N + 3,\, 9,\, N + 4,\, N + 5,\, N + 6,\, 12,\, N + 7,\, N + 8,\, N + 9,\, 15,\, N + 10,\, \\&N + 11,\, 17,\, N + 13,\, 18,\, N + 13,\, N + 15,\, N + 16,\, 22,\, 21,\, 2N + 11.
\end{align*}
See Appendix~\ref{app:Qwd} for explicit computations of these terms, along with a bound on the values of $N$ for which that computation and all previous computations are valid.  In particular, note that the calculations are valid for $N\geq9$.

The last term we have is $\Rh\p{N+24}=2N+11$.  We try to compute $\Rh_N\p{N+25}$:
\begin{align*}
    \Rh_N\p{N+25}&=\Rh_N\p{N+25-\Rh_N\p{N+24}}+\Rh_N\p{N+25-\Rh_N\p{N+23}}\\
    &\phantom{=}+\Rh_N\p{N+25-\Rh_N\p{N+22}}\\
    &=\Rh_N\p{N+25-\pb{2N+11}}+\Rh_N\p{N+25-21}+\Rh_N\p{N+25-22}\\
    &=\Rh_N\pb{-N+14}+\Rh_N\pb{N+4}+\Rh_N\pb{N+3}.
    \end{align*}
If $N\geq14$, then $-N+14\leq0$, so $\Rh_N\p{-N+14}$ is undefined and the sequence dies, as required.

\end{proof}

Theorem~\ref{thm:1thruN} says that $\Rh_N$ dies for all but five values $N$.  The sequences $\Rh_5$ and $\Rh_6$ are identical; both are the $\Rh$-sequence~\cite{slowtrihof}. Sequences $B_7$, $B_8$, and $B_9$ are more akin to Hofstadter's Q sequence.  Like Hofstadter's, it is unclear whether any of these sequences die.  All last for at least $30$ million terms~\cite{oeis}*{A373227, A373228, A373229}.
Plots of the first hundred thousand terms of each of these sequences are shown in Figure~\ref{fig:b789}.

\begin{figure}
\begin{multicols}{3}
\includegraphics[width=2.3in]{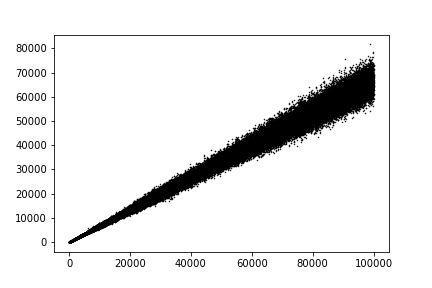}

\includegraphics[width=2.3in]{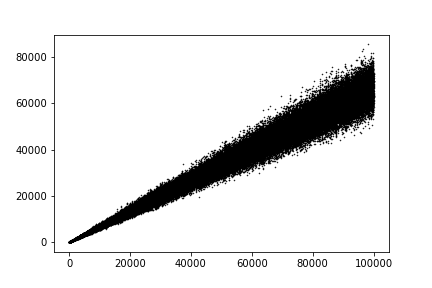}

\includegraphics[width=2.3in]{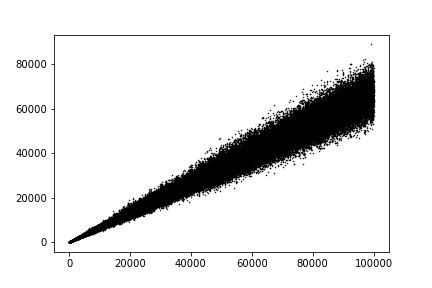}
\end{multicols}
\caption{Plots of the first 100,000 terms of $\Rh_7$ (left), $\Rh_8$ (center), and $\Rh_9$ (right)}
\label{fig:b789}
\end{figure}

\section{Linear Initial Conditions with Extra Zeroes}\label{s:sd}
The sequences in Section~\ref{s:wd} almost all die.  As in~\cite{hof1thruN}, we now consider what happens if we prevent them from dying quickly by defining their values to be zero at nonpositive integers.  For an integer $N\geq3$, let $\Rh_{\bar{N}}$ denote the sequence obtained from the $\Rh$-recurrence with initial conditions $\abk{\bar{0};1,2,3,\ldots,N}$.

In~\cite{hof1thruN}, it was seen that the corresponding behavior for the Hofstadter $Q$-recurrences depends on the congruence class of $N$ modulo~$5$. Three of those cases lead to the end of the sequence, one leads to a semi-predictable pattern that seems to go on forever, and the fifth case leads to a dependence on the congruence class of $N$ modulo~$25$ and, thereafter, potentially higher powers of $5$. For the sequences $\Rh_{\bar{N}}$, the dependence is instead on the congruence class of $N$ modulo~$7$. Here, all seven cases lead to the end of the sequence without needing to consider any cases involving higher powers of $7$. But, all cases require fairly large values of $N$ to be valid.
We have the following main result.

\begin{theorem}\label{thm:trihof1N}
Let $N\geq72$ be 
a 
natural number.  Then the following period-$7$ pattern begins at index $N+67$ in $\Rh_{\bar{N}}$, where $k$ denotes a positive integer:
\[
\begin{cases}
\Rh_{\bar{N}}\p{N+7k}=7k+2\\
\Rh_{\bar{N}}\p{N+7k+1}=N+7k+2\\
\Rh_{\bar{N}}\p{N+7k+2}=N+7k+4\\
\Rh_{\bar{N}}\p{N+7k+3}=7\\
\Rh_{\bar{N}}\p{N+7k+4}=2N+2k+45\\
\Rh_{\bar{N}}\p{N+7k+5}=2N+k-7\\
\Rh_{\bar{N}}\p{N+7k+6}=N-2.
\end{cases}
\]
This pattern lasts through index $2N+\nu$, where 
\[
\nu=\begin{cases}
	-1&\text{if }N\equiv0\pmod{7}\\
	-2&\text{if }N\equiv1\pmod{7}\\
	-2&\text{if }N\equiv2\pmod{7}\\
	-2&\text{if }N\equiv3\pmod{7}\\
	2&\text{if }N\equiv4\pmod{7}\\
	1&\text{if }N\equiv5\pmod{7}\\
	0&\text{if }N\equiv6\pmod{7}.
\end{cases}
\]
After this,
\begin{itemize}
\item If $N\equiv 0\pmod{7}$ and $N\geq196$, then $\Rh_{\bar{N}}$ ends after $2N+27$ terms.
\item If $N\equiv 1\pmod{7}$ and $N\geq2087$, then $\Rh_{\bar{N}}$ ends after $2N+254$ terms.
\item If $N\equiv 2\pmod{7}$ and $N\geq3201$, then $\Rh_{\bar{N}}$ ends after $2N+524$ terms.
\item If $N\equiv 3\pmod{7}$ and $N\geq4315$, then $\Rh_{\bar{N}}$ ends after $2N+560$ terms.
\item If $N\equiv 4\pmod{7}$ and $N\geq200$, then $\Rh_{\bar{N}}$ ends after $2N+20$ terms.
\item If $N\equiv 5\pmod{7}$ and $N\geq32478$, then $\Rh_{\bar{N}}$ ends after $2N+4547$ terms.
\item If $N\equiv 6\pmod{7}$ and $N\geq118$, then $\Rh_{\bar{N}}$ ends after $2N+9$ terms.
\end{itemize}
\end{theorem}

The proof of Theorem~\ref{thm:trihof1N} uses the following lemma.
\begin{lemma}\label{lem:7cyc}
Let $K\geq 7$, $c\geq1$ and $0\leq\Cs\leq 6$. Then, let $\As$ and $\Bs$ be integers such that
\[
\As\geq\begin{cases}
-2c+2&\text{if }\Cs=0\\
-2c+1&\text{if }\Cs=1\\
-2c+4&\text{if }\Cs=2\\
-2c+3&\text{if }\Cs=3\\
-2c+2&\text{if }\Cs=4\\
-2c+1&\text{if }\Cs=5\\
-2c&\text{if }\Cs=6
\end{cases}
\hspace{0.2in}\text{and}\hspace{0.2in}
\Bs\geq\begin{cases}
-c&\text{if }\Cs=0\\
-c&\text{if }\Cs=1\\
-c+3&\text{if }\Cs=2\\
-c+2&\text{if }\Cs=3\\
-c+1&\text{if }\Cs=4\\
-c&\text{if }\Cs=5\\
-c-1&\text{if }\Cs=6.
\end{cases}
\]
Define $L=K-7c-\Cs$ and $M=K+L+5$.
Then, for arbitrary integers $a_1,a_2,\ldots,a_L$, denote the sequence resulting from the $\Rh$-recurrence and the 
initial conditions
\[
\abk{\bar{0};1,2,\ldots,K,6,a_1,a_2,\ldots,a_L,2K+\As-2,2K+\Bs-1,K-2}
\]
 by $\Rh_C$.  
The sequence $\Rh_C$ follows the following pattern from $\Rh_C\p{M-3}$ through $B_C\p{2K+\nu}$
\[
\begin{cases}
\Rh_C\p{M+7k}=L+7k+7\\
\Rh_C\p{M+7k+1}=M+7k+2\\
\Rh_C\p{M+7k+2}=M+7k+4\\
\Rh_C\p{M+7k+3}=7\\
\Rh_C\p{M+7k+4}=2K+2k+\As\\
\Rh_C\p{M+7k+5}=2K+k+\Bs\\
\Rh_C\p{M+7k+6}=K-2
\end{cases}
\]
where
\[
\nu=\begin{cases}
-2&\text{if }\Cs=0\\
-2&\text{if }\Cs=1\\
2&\text{if }\Cs=2\\
1&\text{if }\Cs=3\\
0&\text{if }\Cs=4\\
-1&\text{if }\Cs=5\\
-2&\text{if }\Cs=6.
\end{cases}
\]
\end{lemma}
\begin{proof}
The proof is by induction on the index.  The base cases are $B_C\p{M-3}$ through $B_C\p{M-1}$, which are part of the initial conditions.  Now, suppose $M\leq n\leq2K+\nu$, 
and suppose that $\Rh_C\p{n'}$ is what we want it to be for all $M-3\leq n'<n$.  
 
There are seven cases to consider:
\begin{description}
\item[$n-M\equiv0\pmod{7}$:] In this case, $n=M+7k$ for some $k$.  Applying the $\Rh$-recurrence, we have
\begin{align*}
\Rh_C\p{M+7k}&=\Rh_C\p{M+7k-\Rh_C\p{M+7k-1}}\\
&\phantom{=}+\Rh_C\p{M+7k-\Rh_C\p{M+7k-2}}\\
&\phantom{=}+\Rh_C\p{M+7k-\Rh_C\p{M+7k-3}}\\
&=\Rh_C\p{M+7k-\pb{K-2}}+\Rh_C\p{M+7k-\pb{2K+k-1+\Bs}}\\
&\phantom{=}+\Rh_C\p{M+7k-\pb{2K+2k-2+\As}}\\
&=\Rh_C\p{L+7k+7}+\Rh_C\p{-K+L+6k+6-\Bs}\\
&\phantom{=}+\Rh_C\p{-K+L+5k+7-\As}.
\end{align*}
We know that $n\leq 2K+\nu$. Since $n=M+7k$, we have $n-M=7k\leq2K+\nu-M=K-L-5+\nu=7c+\Cs-5+\nu$. Observe that
\[
\Cs-5+\nu=\begin{cases}
-7 & \Cs=0\\
-6 & \Cs=1\\
-1 & \text{otherwise.}
\end{cases}
\]
Since $7k\equiv0\pmod{7}$, we actually have $7k\leq 7c-7$, since $\fl{\frac{\Cs-5+\nu}{7}}$ always equals $-1$.
In particular, this means that $L+7k+7\leq L+7c=K-\Cs\leq K$. 
As a result, $\Rh_C\p{L+7k+7}=L+7k+7$.

We also have
\[
-K+L+6k+6-\Bs=-7c-\Cs+6k+6-\Bs\leq-7c-\Cs+\frac{6}{7}\p{7c-7}+6-\Bs=-c-\Cs-\Bs.
\]
Observe that
\[
-c-\Cs-\Bs=\begin{cases}
0&\Cs=0\\
-1&\Cs=1\\
-5&\text{otherwise,}
\end{cases}
\]
which implies that $-c-\Cs-\Bs\leq0$. In turn, this means that $-K+L+6k+6-\Bs\leq0$, implying that $\Rh_C\p{-K+L+6k+6-\Bs}=0$.

Similarly, we have that
\[
-K+L+5k+7-\As=-7c-\Cs+5k+7-\As\leq-7c-\Cs+\frac{5}{7}\p{7c-7}+7-\As=-2c-\Cs+2-\As.
\]
Observe that
\[
-2c-\Cs+2-\As=\begin{cases}
0&\Cs=0\text{ or }\Cs=1\\
-4&\text{otherwise,}
\end{cases}
\]
which implies that $-2c-\Cs+2-\As\leq0$. In turn, this means that $-K+L+5k+7-\As\leq0$, implying that $\Rh_C\p{-K+L+5k+7-\As}=0$.
So, $\Rh_C\p{M+7k}=L+7k+7$, as required.

\item[$n-M\equiv1\pmod{7}$:] In this case, $n=M+7k+1$ for some $k$.  Applying the $\Rh$-recurrence, we have
\begin{align*}
\Rh_C\p{M+7k+1}&=\Rh_C\p{M+7k+1-\Rh_C\p{M+7k}}\\
&\phantom{=}+\Rh_C\p{M+7k+1-\Rh_C\p{M+7k-1}}\\
&\phantom{=}+\Rh_C\p{M+7k+1-\Rh_C\p{M+7k-2}}\\
&=\Rh_C\p{M+7k+1-\pb{L+7k+7}}+\Rh_C\p{M+7k+1-\pb{K-2}}\\
&\phantom{=}+\Rh_C\p{M+7k+1-\pb{2K+k-1+\Bs}}\\
&=\Rh_C\p{K-1}+\Rh_C\p{L+7k+8}+\Rh_C\p{-K+L+6k+7-\Bs}.
\end{align*}
We know that $\Rh_C\p{K-1}=K-1$. 
We also know that $n\leq 2K+\nu$. Since $n=M+7k+1$, we have $n-M-1=7k\leq2K+\nu-M-1=K-L-6+\nu=7c+\Cs-6+\nu$. Observe that
\[
\Cs-6+\nu=\begin{cases}
-8 & \Cs=0\\
-7 & \Cs=1\\
-2 & \text{otherwise.}
\end{cases}
\]
Since $7k\equiv0\pmod{7}$, we actually have
\[
7k\leq
\begin{cases}
7c-14 & \Cs=0\\
7c-7 & \text{otherwise.}
\end{cases}
\]
In particular, this means that if $\Cs=0$ then $L+7k+8\leq L+7c-6=K-6<K$, and if $\Cs\neq0$ then $L+7k+8\leq L+7c+1=K-\Cs+1\leq K$. 
As a result, $\Rh_C\p{L+7k+8}=L+7k+8$.

We also have $-K+L+6k+7-\Bs=-7c-\Cs+6k+7-\Bs$.
When $\Cs=0$, $-7c-\Cs+6k+7-\Bs\leq-7c-\Cs+\frac{6}{7}\p{7c-14}+7-\Bs=-c-\Cs-5-\Bs=-5<0.$ When $\Cs\neq0$,
$-7c-\Cs+6k+7-\Bs\leq-7c-\Cs+\frac{6}{7}\p{7c-7}+7-\Bs=-c-\Cs+1-\Bs.$
Observe that
\[
-c-\Cs+1-\Bs=\begin{cases}
1&\Cs=0\\
0&\Cs=1\\
-4&\text{otherwise.}
\end{cases}
\]
Putting all cases  together imply that $-c-\Cs+1-\Bs\leq0$ when $\Cs\neq0$. In turn, this all means that $-K+L+6k+7-\Bs\leq0$, implying that $\Rh_C\p{-K+L+6k+7-\Bs}=0$.
So, $\Rh_C\p{M+7k+1}=\pb{K-1}+\pb{L+7k+8}=K+L+7k+7=M+7k+2$, as required.
\item[$n-M\equiv2\pmod{7}$:] In this case, $n=M+7k+2$ for some $k$.  Applying the $\Rh$-recurrence, we have
\begin{align*}
\Rh_C\p{M+7k+2}&=\Rh_C\p{M+7k+2-\Rh_C\p{M+7k+1}}\\
&\phantom{=}+\Rh_C\p{M+7k+2-\Rh_C\p{M+7k}}\\
&\phantom{=}+\Rh_C\p{M+7k+2-\Rh_C\p{M+7k-1}}\\
&=\Rh_C\p{M+7k+2-\pb{M+7k+2}}\\
&\phantom{=}+\Rh_C\p{M+7k+2-\pb{L+7k+7}}+\Rh_C\p{M+7k+2-\pb{K-2}}\\
&=\Rh_C\p{0}+\Rh_C\p{K}+\Rh_C\p{L+7k+9}.
\end{align*}
We know that $\Rh_C\p{0}=0$ and that $\Rh_C\p{K}=K$. 
We also know that $n\leq 2K+\nu$. 
Since $n=M+7k+2$, we have $n-M-2=7k\leq2K+\nu-M-2=K-L-7+\nu=7c+\Cs-7+\nu$. Observe that
\[
\Cs-7+\nu=\begin{cases}
-9 & \Cs=0\\
-8 & \Cs=1\\
-3 & \text{otherwise.}
\end{cases}
\]
Since $7k\equiv0\pmod{7}$, we actually have
\[
7k\leq
\begin{cases}
7c-14 & \Cs=0\text{ or }\Cs=1\\
7c-7 & \text{otherwise.}
\end{cases}
\]
In particular, this means that if $\Cs\leq1$ then $L+7k+9\leq L+7c-4=K-\Cs-4<K$, and if $\Cs\geq2$ then $L+7k+9\leq L+7c+2=K-\Cs+2\leq K$. 
As a result, $\Rh_C\p{L+7k+9}=L+7k+9$.
So, $\Rh_C\p{M+7k+2}=K+L+7k+9=M+7k+4$, as required.
\item[$n-M\equiv3\pmod{7}$:] In this case, $n=M+7k+3$ for some $k$.  Applying the $\Rh$-recurrence, we have
\begin{align*}
\Rh_C\p{M+7k+3}&=\Rh_C\p{M+7k+3-\Rh_C\p{M+7k+2}}\\
&\phantom{=}+\Rh_C\p{M+7k+3-\Rh_C\p{M+7k+1}}\\
&\phantom{=}+\Rh_C\p{M+7k+3-\Rh_C\p{M+7k}}\\
&=\Rh_C\p{M+7k+3-\pb{M+7k+4}}\\
&\phantom{=}+\Rh_C\p{M+7k+3-\pb{M+7k+2}}\\
&\phantom{=}+\Rh_C\p{M+7k+3-\pb{L+7k+7}}\\
&=\Rh_C\p{-1}+\Rh_C\p{1}+\Rh_C\p{K+1}=0+1+6=7,
\end{align*}
as required.
\item[$n-M\equiv4\pmod{7}$:] In this case, $n=M+7k+4$ for some $k$.  Applying the $\Rh$-recurrence, we have
\begin{align*}
\Rh_C\p{M+7k+4}&=\Rh_C\p{M+7k+4-\Rh_C\p{M+7k+3}}\\
&\phantom{=}+\Rh_C\p{M+7k+4-\Rh_C\p{M+7k+2}}\\
&\phantom{=}+\Rh_C\p{M+7k+4-\Rh_C\p{M+7k+1}}\\
&=\Rh_C\p{M+7k+4-7}+\Rh_C\p{M+7k+4-\pb{M+7k+4}}\\
&\phantom{=}+\Rh_C\p{M+7k+4-\pb{M+7k+2}}\\
&=\Rh_C\p{M+7k-3}+\Rh_C\p{0}+\Rh_C\p{2}.
\end{align*}
We know that $\Rh_C\p{0}=0$ and that $\Rh_C\p{2}=2$. By induction, we have $\Rh_C\p{M+7k-3}=2K+2k+\As-2$. So, $\Rh_C\p{M+7k+4}=2K+2k+\As$, as required.
\item[$n-M\equiv5\pmod{7}$:] In this case, $n=M+7k+5$ for some $k$.  Applying the $\Rh$-recurrence, we have
\begin{align*}
\Rh_C\p{M+7k+5}&=\Rh_C\p{M+7k+5-\Rh_C\p{M+7k+4}}\\
&\phantom{=}+\Rh_C\p{M+7k+5-\Rh_C\p{M+7k+3}}\\
&\phantom{=}+\Rh_C\p{M+7k+5-\Rh_C\p{M+7k+2}}\\
&=\Rh_C\p{M+7k+5-\pb{2K+2k+\As}}+\Rh_C\p{M+7k+5-7}\\
&\phantom{=}+\Rh_C\p{M+7k+5-\pb{M+7k+4}}\\
&=\Rh_C\p{-K+L+5k+10-\As}+\Rh_C\p{M+7k-2}+\Rh_C\p{1}.
\end{align*}
We know that $\Rh_C\p{1}=1$. By induction, we have $\Rh_C\p{M+7k-2}=2K+k+\Bs-1$. 
We also know that $n\leq 2K+\nu$. 
Since $n=M+7k+5$, we have $n-M-5=7k\leq2K+\nu-M-5=K-L-10+\nu=7c+\Cs-10+\nu$. Observe that
\[
\Cs-10+\nu=\begin{cases}
-12 & \Cs=0\\
-11 & \Cs=1\\
-6 & \text{otherwise.}
\end{cases}
\]
Since $7k\equiv0\pmod{7}$, we actually have
\[
7k\leq
\begin{cases}
7c-14 & \Cs=0\text{ or }\Cs=1\\
7c-7 & \text{otherwise.}
\end{cases}
\]
From here, we have $-K+L+5k+10-\As=-7c-\Cs+5k+10-\As$.
When $\Cs\leq1$, $-7c-\Cs+5k+10-\As\leq-7c-\Cs+\frac{5}{7}\p{7c-14}+10-\As=-2c-\Cs-\As<0.$ When $\Cs\geq2$,
$-7c-\Cs+5k+10-\As\leq-7c-\Cs+\frac{5}{7}\p{7c-7}+10-\As=-2c-\Cs+5-\As.$
Observe that $-2c-\Cs+5-\As=-1$ when $\Cs\geq2$, which is less than $0$. In turn, this all means that $-K+L+5k+10-\As\leq0$ regardless of value of $\Cs$, implying that $\Rh_C\p{-K+L+5k+10-\As}=0$. 
Therefore $\Rh_C\p{M+7k+5}=2K+k+\Bs$, as required.
\item[$n-M\equiv6\pmod{7}$:] In this case, $n=M+7k+6$ for some $k$.  Applying the $\Rh$-recurrence, we have
\begin{align*}
\Rh_C\p{M+7k+6}&=\Rh_C\p{M+7k+6-\Rh_C\p{M+7k+5}}\\
&\phantom{=}+\Rh_C\p{M+7k+6-\Rh_C\p{M+7k+4}}\\
&\phantom{=}+\Rh_C\p{M+7k+6-\Rh_C\p{M+7k+3}}\\
&=\Rh_C\p{M+7k+6-\pb{2K+k+\Bs}}\\
&\phantom{=}+\Rh_C\p{M+7k+6-\pb{2K+2k+\As}}+\Rh_C\p{M+7k+6-7}\\
&=\Rh_C\p{-K+L+6k+11-\Bs}+\Rh_C\p{-K+L+5k+11-\As}\\
&\phantom{=}+\Rh_C\p{M+7k-1}.
\end{align*}
By induction, we have $\Rh_C\p{M+7k-1}=K-2$. 
We also know that $n\leq 2K+\nu$. 
Since $n=M+7k+6$, we have $n-M-6=7k\leq2K+\nu-M-6=K-L-11+\nu=7c+\Cs-11+\nu$. Observe that
\[
\Cs-11+\nu=\begin{cases}
-13 & \Cs=0\\
-12 & \Cs=1\\
-7 & \text{otherwise.}
\end{cases}
\]
Since $7k\equiv0\pmod{7}$, we actually have
\[
7k\leq
\begin{cases}
7c-14 & \Cs=0\text{ or }\Cs=1\\
7c-7 & \text{otherwise.}
\end{cases}
\]
From here, we have $-K+L+6k+11-\Bs=-7c-\Cs+6k+11-\Bs$.
When $\Cs\leq1$, $-7c-\Cs+6k+11-\Bs\leq-7c-\Cs+\frac{6}{7}\p{7c-14}+11-\Bs=-c-\Cs-1-\Bs<0.$ When $\Cs\geq2$,
$-7c-\Cs+6k+11-\Bs\leq-7c-\Cs+\frac{6}{7}\p{7c-7}+11-\Bs=-c-\Cs+5-\Bs.$
Observe that $-c-\Cs+5-\Bs=0$ when $\Cs\geq2$. In turn, this all means that $-K+L+6k+11-\Bs\leq0$ regardless of value of $\Cs$, implying that $\Rh_C\p{-K+L+6k+11-\Bs}=0$.
Similarly, we have $-K+L+5k+11-\As=-7c-\Cs+5k+11-\As$.
When $\Cs\leq1$, $-7c-\Cs+5k+11-\As\leq-7c-\Cs+\frac{5}{7}\p{7c-14}+11-\As=-2c-\Cs+1-\As<0.$ When $\Cs\geq2$,
$-7c-\Cs+5k+11-\As\leq-7c-\Cs+\frac{5}{7}\p{7c-7}+11-\As=-2c-\Cs+6-\As.$
Observe that $-2c-\Cs+6-\As=0$ when $\Cs\geq2$. In turn, this all means that $-K+L+5k+11-\As\leq0$ regardless of value of $\Cs$, implying that $\Rh_C\p{-K+L+5k+11-\As}=0$.
Therefore, $\Rh_C\p{M+7k+6}=K-2$, as required.
\end{description}
\end{proof}

We now prove Theorem~\ref{thm:trihof1N}.
\begin{proof}
We refer the reader to Appendix~\ref{app:Qwd} for terms $\Rh_{\bar N}\p{1}$ through $\Rh_{\bar N}\p{N+28}$.  Those calculations, which are for $\Rh_N$, also apply to $\Rh_{\bar N}$.  From there, Appendix~\ref{app:Qbwd} calculates and lists terms $\Rh_{\bar N}\p{N+25}$ through $\Rh_{\bar N}\p{N+69}$. Those calculations are all valid provided $N\geq67$. 
Observe that the values $K=N$, $c=\fl{\frac{N-65}{7}}$, $\Cs=\pb{N-65}\bmod 7$, $\As=65$, and $\Bs=3$ satisfy the conditions of Lemma~\ref{lem:7cyc} provided that $N\geq72$ (so that $c\geq1$), and the first $N+69$ terms of $\Rh_{\bar N}$ can be used as initial conditions as per that lemma. Keeping in mind that these choices of parameters mean that $L=65$ and $M=N+70$, Lemma~\ref{lem:7cyc} implies we have the following pattern for $B_{\bar N}$ for some time:
\[
\begin{cases}
\Rh_{\bar N}\p{N+70+7k'}=65+7k'+7=7k'+72\\
\Rh_{\bar N}\p{N+70+7k'+1}=N+70+7k'+2=N+7k'+72\\
\Rh_{\bar N}\p{N+70+7k'+2}=N+70+7k'+4=N+7k'+74\\
\Rh_{\bar N}\p{N+70+7k'+3}=7\\
\Rh_{\bar N}\p{N+70+7k'+4}=2N+2k'+65\\
\Rh_{\bar N}\p{N+70+7k'+5}=2N+k'+3\\
\Rh_{\bar N}\p{N+70+7k'+6}=N-2.
\end{cases}
\]
Re-indexing so that $k=k'+10$ allows us to rewrite this pattern as
\[
\begin{cases}
\Rh_{\bar N}\p{N+7k}=7\pb{k-10}+72=7k+2\\
\Rh_{\bar N}\p{N+7k+1}=N+7\pb{k-10}+72=N+7k+2\\
\Rh_{\bar N}\p{N+7k+2}=N+7\pb{k-10}+74=N+7k+4\\
\Rh_{\bar N}\p{N+7k+3}=7\\
\Rh_{\bar N}\p{N+7k+4}=2N+2\pb{k-10}+65=2N+2k+45\\
\Rh_{\bar N}\p{N+7k+5}=2N+\pb{k-10}+3=2N+k-7\\
\Rh_{\bar N}\p{N+7k+6}=N-2,
\end{cases}
\]
which is the desired pattern. Lemma~\ref{lem:7cyc} guarantees that this pattern persists through index $2N+\nu$, where $\nu$ depends on $\Cs$, which, in turn, depends on $N$. Here, we have
\[
\nu=\begin{cases}
-1 & \text{if }N\equiv 0\pmod{7}\\
-2 & \text{if }N\equiv 1\pmod{7}\\
-2 & \text{if }N\equiv 2\pmod{7}\\
-2 & \text{if }N\equiv 3\pmod{7}\\
2 & \text{if }N\equiv 4\pmod{7}\\
1 & \text{if }N\equiv 5\pmod{7}\\
0 & \text{if }N\equiv 6\pmod{7},
\end{cases}
\]
as required.

We now prove the remainder of Theorem~\ref{thm:trihof1N}, regarding the ways in which $B_{\bar N}$ can end.  For each possibility of $N\bmod 7$, we can compute terms of $B_{\bar N}$ from index $2N+\nu+1$ onward, using the now known values of $B_{\bar N}$ for all smaller indices. These computations are akin to those in the appendices, and like those we track how large $N$ needs to be for the computations to be valid. The result is a lengthy and tedious list of terms and bounds, but the claimed end conditions in the statement of the theorem are all validated. The full length of the computations can be found on \href{https://github.com/nhf216/B-recurrence-data}{GitHub}\footnote{\url{https://github.com/nhf216/B-recurrence-data}}, with one file for each value of $N\bmod 7$.
\end{proof}

The sequences corresponding to the minimum values of $N$ for each congruence class in Theorem~\ref{thm:trihof1N} are all available in OEIS~\cite{oeis}*{A373234, A373235, A373236, A373237, A373238, A274058, A373239}.


\subsection{The Remaining Values of $N$}\label{ss:sporadic}

Theorem~\ref{thm:trihof1N} characterizes the behavior of $B_{\bar N}$ for all but $6079$ values of $N\geq3$. 
These sequences can be studied individually by generating the sequences and observing the terms.  This study is carried out in~\cite{thesis}; what follows is a summary of those findings.  All of these sequences end before $150$ million terms except when
\begin{align*}
N\in\{4,&5,6,7,8,9,10,11,12,13,14,15,18,81,182,193,429,822,1892,2789,3442,7292,\\
&20830,23511,25163\}.
\end{align*}
 Of these $B_{\bar 5}$ and $B_{\bar 6}$ are the $B$-sequence, so they last forever. For
 \[
 N\in\st{4, 7, 8, 9, 10, 11, 12, 13, 14, 15, 18},
 \]
 $B_{\bar N}$ appears to last forever but exhibits chaotic behavior, akin to the sequences in Figure~\ref{fig:b789}. To discuss the remaining values, we need four more results like Lemma~\ref{lem:7cyc}. All come from~\cite{thesis} and are stated without proof. All proofs are straightforward but tedious inductive arguments, much like the proof of Lemma~\ref{lem:7cyc}. Here, unlike in Lemma~\ref{lem:7cyc}, the bounds on the parameters are generally not optimized.

\begin{lemma}\label{lem:2cyc}
Let $K\geq 1$ and $M\geq K+5$ be integers.
Then, for arbitrary integers $a_1,a_2,\ldots,a_K$, denote the sequence resulting from the $\Rh$-recurrence and the 
initial conditions
\[
\abk{\bar{0};a_1,a_2,\ldots,a_K,2,M,2}
\]
 by $\Rh_D$. 
The sequence $\Rh_D$ follows the following pattern starting from $\Rh_D\p{K+1}$ (and lasting forever):
\[
\begin{cases}
\Rh_D\p{K+2k}=2^{k-1}\cdot M\\
\Rh_D\p{K+2k+1}=2.
\end{cases}
\]
\end{lemma}

\begin{lemma}\label{lem:5cyc}
Let $K\geq 3$, and $\Bs\geq1$ be integers.
Then, for arbitrary integers $a_4,a_5,\ldots,a_K$, denote the sequence resulting from the $\Rh$-recurrence and the 
initial conditions
\[
\abk{\bar{0};1,2,3,a_4,a_5,\ldots,a_K,K+\Bs,3,K+3,K+\Bs+1,5}
\]
 by $\Rh_E$.  
The sequence $\Rh_E$ follows the following pattern from $\Rh_E\p{K+1}$ through $\Rh_E\!\pb{K+\fl{\frac{5\Bs-15}{2}}}$:
\[
\begin{cases}
\Rh_E\p{K+5k}=5\\
\Rh_E\p{K+5k+1}=K+3k+\Bs\\
\Rh_E\p{K+5k+2}=3\\
\Rh_E\p{K+5k+3}=K+5k+3\\
\Rh_E\p{K+5k+4}=K+3k+\Bs+1.
\end{cases}
\]
\end{lemma}

\begin{lemma}\label{lem:16cyc1}
Let $K$, $\As$, $\Bs_1$, $\Bs_2$, and $\Cs$ be positive integers with $\As>31+K$, $\Bs_1>\As$, $\Bs_2>\As$, and $\Cs>\As$.
Then, for arbitrary integers $a_1,a_2,\ldots,a_K$, denote the sequence resulting from the $\Rh$-recurrence and the 
initial conditions
\[
\abk{\bar{0};a_1,a_2,\ldots,a_K,\As,7,\Bs_2,16,\Bs_2,16,\Bs_1,\As,7,\Bs_2,16,2\Bs_2,16,\Bs_2,25,\Cs,\As,7}
\]
 by $\Rh_T$.  
The sequence $\Rh_T$ follows the following pattern from $\Rh_T\p{K+1}$ through $\Rh_T(\As)$:
\[
\begin{cases}
\Rh_T\p{K+16k}=\Bs_1\cdot2^{k-1}+\Cs-\Bs_1\\
\Rh_T\p{K+16k+1}=\As\\
\Rh_T\p{K+16k+2}=7\\
\Rh_T\p{K+16k+3}=\Bs_2\cdot2^k\\
\Rh_T\p{K+16k+4}=16\\
\Rh_T\p{K+16k+5}=\Bs_2\cdot2^k\\
\Rh_T\p{K+16k+6}=16\\
\Rh_T\p{K+16k+7}=\Bs_1\cdot2^k
\end{cases}
\quad
\begin{cases}
\Rh_T\p{K+16k+8}=\As\\
\Rh_T\p{K+16k+9}=7\\
\Rh_T\p{K+16k+10}=\Bs_2\cdot2^k\\
\Rh_T\p{K+16k+11}=16\\
\Rh_T\p{K+16k+12}=\Bs_2\cdot2^k\\
\Rh_T\p{K+16k+13}=16\\
\Rh_T\p{K+16k+14}=\Bs_2\cdot2^k\\
\Rh_T\p{K+16k+15}=25.
\end{cases}
\]
\end{lemma}

\begin{lemma}\label{lem:16cyc2}
Let $K$, $\As$, $\Bs_1$, $\Bs_2$, $\Cs_1$, $\Cs_2$, and $\Cs_3$ be positive integers with $\As>31+K$, $\Bs_1>\As$, $\Bs_2>\As$, $\Cs_1>\As$,  $\Cs_2>\As$, and $\Cs_3>\As$.
Then, for arbitrary integers $a_1,a_2,\ldots,a_K$, denote the sequence resulting from the $\Rh$-recurrence and the 
initial conditions
\[
\abk{\bar{0};a_1,a_2,\ldots,a_K,16,\Bs_2,7,\Cs_2,\As,16,\As,16,\Bs_1,10,\Cs_3,\Bs_2,7,\As,16,\Cs_1}
\]
 by $\Rh_U$.  
The sequence $\Rh_U$ follows the following pattern from $\Rh_U\p{K+1}$ through $\Rh_U(\As)$:
\[
\begin{cases}
\Rh_U\p{K+16k}=\Bs_1\cdot2^k+\Cs_1-2\Bs_1\\
\Rh_U\p{K+16k+1}=16\\
\Rh_U\p{K+16k+2}=\Bs_2\cdot2^k\\
\Rh_U\p{K+16k+3}=7\\
\Rh_U\p{K+16k+4}=7k+\Cs_2\\
\Rh_U\p{K+16k+5}=\As\\
\Rh_U\p{K+16k+6}=16\\
\Rh_U\p{K+16k+7}=\As
\end{cases}
\quad
\begin{cases}
\Rh_U\p{K+16k+8}=16\\
\Rh_U\p{K+16k+9}=\Bs_1\cdot2^k\\
\Rh_U\p{K+16k+10}=10\\
\Rh_U\p{K+16k+11}=16k+\Cs_3\\
\Rh_U\p{K+16k+12}=\Bs_2\cdot2^k\\
\Rh_U\p{K+16k+13}=7\\
\Rh_U\p{K+16k+14}=\As\\
\Rh_U\p{K+16k+15}=16.
\end{cases}
\]
\end{lemma}

For $N\in\st{81,182,429,822,1892,2789,7292,23511,25163}$, Lemma~\ref{lem:2cyc} eventually applies, so these sequences continue forever. Both $B_{\overline{193}}$ and $B_{\overline{3442}}$ also continue forever. Infinitely many prefixes of these sequences satisfy the hypotheses of Lemma~\ref{lem:5cyc}. In other words, these sequences consist of infinitely many chunks of the sort described by Lemma~\ref{lem:5cyc} with some sporadic terms in between. Each such chunk lasts approximately six times as long as the previous one. The first ten million terms of $\Rh_{\overline{193}}$ are shown in Figure~\ref{fig:b193}.
The only remaining sequence is $B_{\overline{20830}}$. This sequence ends, but it has a total of $84975\cdot2^{560362}+31$ terms, far too many to compute. First, some initial terms of $B_{\overline{20830}}$ satisfy the conditions of Lemma~\ref{lem:16cyc1}, so that lemma governs the behavior of the sequence for awhile. Then, shortly after the chunk of $B_{\overline{20830}}$ described by Lemma~\ref{lem:16cyc1} concludes, Lemma~\ref{lem:16cyc2} applies. That lemma then governs the behavior of the sequence for a long time. That chunk terminates at index $84975\cdot2^{560362}$, after which the sequence lasts only $31$ additional terms~\cite{oeis}*{A283887}. 

\begin{figure}
	\begin{multicols}{2}
		\includegraphics[width=3.5in]{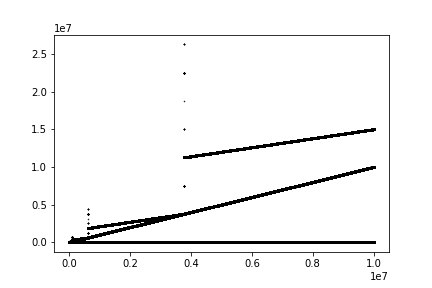}
		
		\includegraphics[width=3.5in]{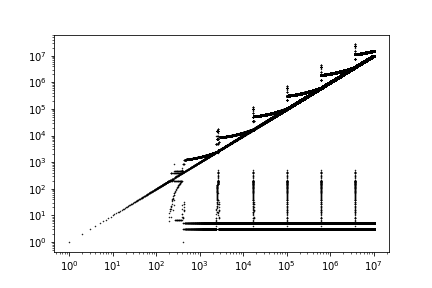}
	\end{multicols}
	\caption{Plots of the first 10,000,000 terms of $\Rh_{\overline{193}}$ with linear scale (left) and logarithmic scale (right)}
	\label{fig:b193}
\end{figure}

\section{Future Work}\label{s:future}
Studying nested recurrence relations with symbolic initial conditions of this type was initiated in~\cite{hof1thruN} with Hofstadter's $Q$-recurrence, and it is continued here with the three-term analog to that recurrence. The obvious idea is to continue adding more terms to the recurrence. But, that study was undertaken in~\cite{thesis}, and it suggests that chaos reigns supreme for the recurrences with four or more terms, even for large $N$. One question, then, would be whether any other symbolic initial conditions lead to predictable solutions to these recurrences. Another direction would be to study these same sorts of initial conditions with other recurrences, such as the Conolly recurrence~\cite{con} or the Tanny recurrence~\cite{tanny}.

The other big research direction suggested by this work is further exploration and discovery of results like Lemmas~\ref{lem:7cyc}, \ref{lem:2cyc}, \ref{lem:5cyc}, \ref{lem:16cyc1}, and~\ref{lem:16cyc2}. These five lemmas all describe temporary or permanent solutions to the $\Rh$-recurrence that consist of interleavings of simpler sequences. These particular results appear in this paper because they are needed to analyze the sequences $\Rh_{\bar{N}}$ for various values of $N$. But, they are by no means exhaustive among results of this type. The first known solution of this type to a nested recurrence is Golomb's solution to the $Q$-recurrence~\cite{golomb}, which uses initial conditions $\abk{3,2,1}$. Later, Ruskey gave a solution to the $Q$-recurrence with exponentially growing subsequences~\cite{rusk} via initial conditions $\abk{\bar{0};3,6,5,3,6,8}$. A general algorithmic framework for finding and proving these sorts of solutions was laid out in~\cite{symbhof}. All of these solutions last forever; the first known occurrence of temporary interleaved solutions is in~\cite{hof1thruN}, in analogous work to this paper on the $Q$-recurrence. Perhaps such solutions could also be worked into the framework from~\cite{symbhof}.

\appendix

\section{First $24$ terms following initial conditions of $\Rh_N$}\label{app:Qwd}

Assuming $N\geq9$, these are the first $24$ terms of $\Rh_N$ following the initial conditions.

\begin{align*}
\mbb{\Rh_N\p{N+1}}&=\Rh_N\p{N+1-\Rh_N\p{N}}+\Rh_N\p{N+1-\Rh_N\p{N-1}}\\
&\phantom{=}+\Rh_N\p{N+1-\Rh_N\p{N-2}}\\
&=\Rh_N\p{N+1-N}+\Rh_N\p{N+1-\p{N - 1}}+\Rh_N\p{N+1-\p{N - 2}}\\
&=\Rh_N\p{1}+\Rh_N\p{2}+\Rh_N\p{3}=1+2+3=\mbb{6}\\
&(N\geq3)
\end{align*}
\begin{align*}
\mbb{\Rh_N\p{N+2}}&=\Rh_N\p{N+2-\Rh_N\p{N+1}}+\Rh_N\p{N+2-\Rh_N\p{N}}\\
&\phantom{=}+\Rh_N\p{N+2-\Rh_N\p{N-1}}\\
&=\Rh_N\p{N+2-6}+\Rh_N\p{N+2-N}+\Rh_N\p{N+2-\p{N - 1}}\\
&=\Rh_N\p{N - 4}+\Rh_N\p{2}+\Rh_N\p{3}=\p{N - 4}+2+3=\mbb{N + 1}\\
&(N\geq5)
\end{align*}
\begin{align*}
\mbb{\Rh_N\p{N+3}}&=\Rh_N\p{N+3-\Rh_N\p{N+2}}+\Rh_N\p{N+3-\Rh_N\p{N+1}}\\
&\phantom{=}+\Rh_N\p{N+3-\Rh_N\p{N}}\\
&=\Rh_N\p{N+3-\p{N + 1}}+\Rh_N\p{N+3-6}+\Rh_N\p{N+3-N}\\
&=\Rh_N\p{2}+\Rh_N\p{N - 3}+\Rh_N\p{3}=2+\p{N - 3}+3=\mbb{N + 2}\\
&(N\geq4)
\end{align*}
\begin{align*}
\mbb{\Rh_N\p{N+4}}&=\Rh_N\p{N+4-\Rh_N\p{N+3}}+\Rh_N\p{N+4-\Rh_N\p{N+2}}\\
&\phantom{=}+\Rh_N\p{N+4-\Rh_N\p{N+1}}\\
&=\Rh_N\p{N+4-\p{N + 2}}+\Rh_N\p{N+4-\p{N + 1}}+\Rh_N\p{N+4-6}\\
&=\Rh_N\p{2}+\Rh_N\p{3}+\Rh_N\p{N - 2}=2+3+\p{N - 2}=\mbb{N + 3}\\
&(N\geq3)
\end{align*}
\begin{align*}
\mbb{\Rh_N\p{N+5}}&=\Rh_N\p{N+5-\Rh_N\p{N+4}}+\Rh_N\p{N+5-\Rh_N\p{N+3}}\\
&\phantom{=}+\Rh_N\p{N+5-\Rh_N\p{N+2}}\\
&=\Rh_N\p{N+5-\p{N + 3}}+\Rh_N\p{N+5-\p{N + 2}}+\Rh_N\p{N+5-\p{N + 1}}\\
&=\Rh_N\p{2}+\Rh_N\p{3}+\Rh_N\p{4}=2+3+4=\mbb{9}\\
&(N\geq4)
\end{align*}
\begin{align*}
\mbb{\Rh_N\p{N+6}}&=\Rh_N\p{N+6-\Rh_N\p{N+5}}+\Rh_N\p{N+6-\Rh_N\p{N+4}}\\
&\phantom{=}+\Rh_N\p{N+6-\Rh_N\p{N+3}}\\
&=\Rh_N\p{N+6-9}+\Rh_N\p{N+6-\p{N + 3}}+\Rh_N\p{N+6-\p{N + 2}}\\
&=\Rh_N\p{N - 3}+\Rh_N\p{3}+\Rh_N\p{4}=\p{N - 3}+3+4=\mbb{N + 4}\\
&(N\geq4)
\end{align*}
\begin{align*}
\mbb{\Rh_N\p{N+7}}&=\Rh_N\p{N+7-\Rh_N\p{N+6}}+\Rh_N\p{N+7-\Rh_N\p{N+5}}\\
&\phantom{=}+\Rh_N\p{N+7-\Rh_N\p{N+4}}\\
&=\Rh_N\p{N+7-\p{N + 4}}+\Rh_N\p{N+7-9}+\Rh_N\p{N+7-\p{N + 3}}\\
&=\Rh_N\p{3}+\Rh_N\p{N - 2}+\Rh_N\p{4}=3+\p{N - 2}+4=\mbb{N + 5}\\
&(N\geq4)
\end{align*}
\begin{align*}
\mbb{\Rh_N\p{N+8}}&=\Rh_N\p{N+8-\Rh_N\p{N+7}}+\Rh_N\p{N+8-\Rh_N\p{N+6}}\\
&\phantom{=}+\Rh_N\p{N+8-\Rh_N\p{N+5}}\\
&=\Rh_N\p{N+8-\p{N + 5}}+\Rh_N\p{N+8-\p{N + 4}}+\Rh_N\p{N+8-9}\\
&=\Rh_N\p{3}+\Rh_N\p{4}+\Rh_N\p{N - 1}=3+4+\p{N - 1}=\mbb{N + 6}\\
&(N\geq4)
\end{align*}
\begin{align*}
\mbb{\Rh_N\p{N+9}}&=\Rh_N\p{N+9-\Rh_N\p{N+8}}+\Rh_N\p{N+9-\Rh_N\p{N+7}}\\
&\phantom{=}+\Rh_N\p{N+9-\Rh_N\p{N+6}}\\
&=\Rh_N\p{N+9-\p{N + 6}}+\Rh_N\p{N+9-\p{N + 5}}+\Rh_N\p{N+9-\p{N + 4}}\\
&=\Rh_N\p{3}+\Rh_N\p{4}+\Rh_N\p{5}=3+4+5=\mbb{12}\\
&(N\geq5)
\end{align*}
\begin{align*}
\mbb{\Rh_N\p{N+10}}&=\Rh_N\p{N+10-\Rh_N\p{N+9}}+\Rh_N\p{N+10-\Rh_N\p{N+8}}\\
&\phantom{=}+\Rh_N\p{N+10-\Rh_N\p{N+7}}\\
&=\Rh_N\p{N+10-12}+\Rh_N\p{N+10-\p{N + 6}}+\Rh_N\p{N+10-\p{N + 5}}\\
&=\Rh_N\p{N - 2}+\Rh_N\p{4}+\Rh_N\p{5}=\p{N - 2}+4+5=\mbb{N + 7}\\
&(N\geq5)
\end{align*}
\begin{align*}
\mbb{\Rh_N\p{N+11}}&=\Rh_N\p{N+11-\Rh_N\p{N+10}}+\Rh_N\p{N+11-\Rh_N\p{N+9}}\\
&\phantom{=}+\Rh_N\p{N+11-\Rh_N\p{N+8}}\\
&=\Rh_N\p{N+11-\p{N + 7}}+\Rh_N\p{N+11-12}+\Rh_N\p{N+11-\p{N + 6}}\\
&=\Rh_N\p{4}+\Rh_N\p{N - 1}+\Rh_N\p{5}=4+\p{N - 1}+5=\mbb{N + 8}\\
&(N\geq5)
\end{align*}
\begin{align*}
\mbb{\Rh_N\p{N+12}}&=\Rh_N\p{N+12-\Rh_N\p{N+11}}+\Rh_N\p{N+12-\Rh_N\p{N+10}}\\
&\phantom{=}+\Rh_N\p{N+12-\Rh_N\p{N+9}}\\
&=\Rh_N\p{N+12-\p{N + 8}}+\Rh_N\p{N+12-\p{N + 7}}+\Rh_N\p{N+12-12}\\
&=\Rh_N\p{4}+\Rh_N\p{5}+\Rh_N\p{N}=4+5+N=\mbb{N + 9}\\
&(N\geq5)
\end{align*}
\begin{align*}
\mbb{\Rh_N\p{N+13}}&=\Rh_N\p{N+13-\Rh_N\p{N+12}}+\Rh_N\p{N+13-\Rh_N\p{N+11}}\\
&\phantom{=}+\Rh_N\p{N+13-\Rh_N\p{N+10}}\\
&=\Rh_N\p{N+13-\p{N + 9}}+\Rh_N\p{N+13-\p{N + 8}}+\Rh_N\p{N+13-\p{N + 7}}\\
&=\Rh_N\p{4}+\Rh_N\p{5}+\Rh_N\p{6}=4+5+6=\mbb{15}\\
&(N\geq6)
\end{align*}
\begin{align*}
\mbb{\Rh_N\p{N+14}}&=\Rh_N\p{N+14-\Rh_N\p{N+13}}+\Rh_N\p{N+14-\Rh_N\p{N+12}}\\
&\phantom{=}+\Rh_N\p{N+14-\Rh_N\p{N+11}}\\
&=\Rh_N\p{N+14-15}+\Rh_N\p{N+14-\p{N + 9}}+\Rh_N\p{N+14-\p{N + 8}}\\
&=\Rh_N\p{N - 1}+\Rh_N\p{5}+\Rh_N\p{6}=\p{N - 1}+5+6=\mbb{N + 10}\\
&(N\geq6)
\end{align*}
\begin{align*}
\mbb{\Rh_N\p{N+15}}&=\Rh_N\p{N+15-\Rh_N\p{N+14}}+\Rh_N\p{N+15-\Rh_N\p{N+13}}\\
&\phantom{=}+\Rh_N\p{N+15-\Rh_N\p{N+12}}\\
&=\Rh_N\p{N+15-\p{N + 10}}+\Rh_N\p{N+15-15}+\Rh_N\p{N+15-\p{N + 9}}\\
&=\Rh_N\p{5}+\Rh_N\p{N}+\Rh_N\p{6}=5+N+6=\mbb{N + 11}\\
&(N\geq6)
\end{align*}
\begin{align*}
\mbb{\Rh_N\p{N+16}}&=\Rh_N\p{N+16-\Rh_N\p{N+15}}+\Rh_N\p{N+16-\Rh_N\p{N+14}}\\
&\phantom{=}+\Rh_N\p{N+16-\Rh_N\p{N+13}}\\
&=\Rh_N\p{N+16-\p{N + 11}}+\Rh_N\p{N+16-\p{N + 10}}+\Rh_N\p{N+16-15}\\
&=\Rh_N\p{5}+\Rh_N\p{6}+\Rh_N\p{N + 1}=5+6+6=\mbb{17}\\
&(N\geq6)
\end{align*}
\begin{align*}
\mbb{\Rh_N\p{N+17}}&=\Rh_N\p{N+17-\Rh_N\p{N+16}}+\Rh_N\p{N+17-\Rh_N\p{N+15}}\\
&\phantom{=}+\Rh_N\p{N+17-\Rh_N\p{N+14}}\\
&=\Rh_N\p{N+17-17}+\Rh_N\p{N+17-\p{N + 11}}+\Rh_N\p{N+17-\p{N + 10}}\\
&=\Rh_N\p{N}+\Rh_N\p{6}+\Rh_N\p{7}=N+6+7=\mbb{N + 13}\\
&(N\geq7)
\end{align*}
\begin{align*}
\mbb{\Rh_N\p{N+18}}&=\Rh_N\p{N+18-\Rh_N\p{N+17}}+\Rh_N\p{N+18-\Rh_N\p{N+16}}\\
&\phantom{=}+\Rh_N\p{N+18-\Rh_N\p{N+15}}\\
&=\Rh_N\p{N+18-\p{N + 13}}+\Rh_N\p{N+18-17}+\Rh_N\p{N+18-\p{N + 11}}\\
&=\Rh_N\p{5}+\Rh_N\p{N + 1}+\Rh_N\p{7}=5+6+7=\mbb{18}\\
&(N\geq7)
\end{align*}
\begin{align*}
\mbb{\Rh_N\p{N+19}}&=\Rh_N\p{N+19-\Rh_N\p{N+18}}+\Rh_N\p{N+19-\Rh_N\p{N+17}}\\
&\phantom{=}+\Rh_N\p{N+19-\Rh_N\p{N+16}}\\
&=\Rh_N\p{N+19-18}+\Rh_N\p{N+19-\p{N + 13}}+\Rh_N\p{N+19-17}\\
&=\Rh_N\p{N + 1}+\Rh_N\p{6}+\Rh_N\p{N + 2}=6+6+\p{N + 1}=\mbb{N + 13}\\
&(N\geq6)
\end{align*}
\begin{align*}
\mbb{\Rh_N\p{N+20}}&=\Rh_N\p{N+20-\Rh_N\p{N+19}}+\Rh_N\p{N+20-\Rh_N\p{N+18}}\\
&\phantom{=}+\Rh_N\p{N+20-\Rh_N\p{N+17}}\\
&=\Rh_N\p{N+20-\p{N + 13}}+\Rh_N\p{N+20-18}+\Rh_N\p{N+20-\p{N + 13}}\\
&=\Rh_N\p{7}+\Rh_N\p{N + 2}+\Rh_N\p{7}=7+\p{N + 1}+7=\mbb{N + 15}\\
&(N\geq7)
\end{align*}
\begin{align*}
\mbb{\Rh_N\p{N+21}}&=\Rh_N\p{N+21-\Rh_N\p{N+20}}+\Rh_N\p{N+21-\Rh_N\p{N+19}}\\
&\phantom{=}+\Rh_N\p{N+21-\Rh_N\p{N+18}}\\
&=\Rh_N\p{N+21-\p{N + 15}}+\Rh_N\p{N+21-\p{N + 13}}+\Rh_N\p{N+21-18}\\
&=\Rh_N\p{6}+\Rh_N\p{8}+\Rh_N\p{N + 3}=6+8+\p{N + 2}=\mbb{N + 16}\\
&(N\geq8)
\end{align*}
\begin{align*}
\mbb{\Rh_N\p{N+22}}&=\Rh_N\p{N+22-\Rh_N\p{N+21}}+\Rh_N\p{N+22-\Rh_N\p{N+20}}\\
&\phantom{=}+\Rh_N\p{N+22-\Rh_N\p{N+19}}\\
&=\Rh_N\p{N+22-\p{N + 16}}+\Rh_N\p{N+22-\p{N + 15}}\\
&\phantom{=}+\Rh_N\p{N+22-\p{N + 13}}\\
&=\Rh_N\p{6}+\Rh_N\p{7}+\Rh_N\p{9}=6+7+9=\mbb{22}\\
&(N\geq9)
\end{align*}
\begin{align*}
\mbb{\Rh_N\p{N+23}}&=\Rh_N\p{N+23-\Rh_N\p{N+22}}+\Rh_N\p{N+23-\Rh_N\p{N+21}}\\
&\phantom{=}+\Rh_N\p{N+23-\Rh_N\p{N+20}}\\
&=\Rh_N\p{N+23-22}+\Rh_N\p{N+23-\p{N + 16}}+\Rh_N\p{N+23-\p{N + 15}}\\
&=\Rh_N\p{N + 1}+\Rh_N\p{7}+\Rh_N\p{8}=6+7+8=\mbb{21}\\
&(N\geq8)
\end{align*}
\begin{align*}
\mbb{\Rh_N\p{N+24}}&=\Rh_N\p{N+24-\Rh_N\p{N+23}}+\Rh_N\p{N+24-\Rh_N\p{N+22}}\\
&\phantom{=}+\Rh_N\p{N+24-\Rh_N\p{N+21}}\\
&=\Rh_N\p{N+24-21}+\Rh_N\p{N+24-22}+\Rh_N\p{N+24-\p{N + 16}}\\
&=\Rh_N\p{N + 3}+\Rh_N\p{N + 2}+\Rh_N\p{8}=\p{N + 2}+\p{N + 1}+8=\mbb{2N + 11}\\
&(N\geq8)
\end{align*}

\section{Terms $25$ through $69$ following initial conditions of $\Rh_{\bar N}$}\label{app:Qbwd}

Assuming that $N\geq67$, these are terms $25$ through $69$ of $\Rh_{\bar N}$ following the initial conditions. (Terms $1$ through $24$ agree with $\Rh_N$ and can be found in Appendix~\ref{app:Qwd}.)

\begin{align*}
\mbb{\Rh_{\bar N}\p{N+25}}&=\Rh_{\bar N}\p{N+25-\Rh_{\bar N}\p{N+24}}+\Rh_{\bar N}\p{N+25-\Rh_{\bar N}\p{N+23}}\\
&\phantom{=}+\Rh_{\bar N}\p{N+25-\Rh_{\bar N}\p{N+22}}\\
&=\Rh_{\bar N}\p{N+25-\p{2N + 11}}+\Rh_{\bar N}\p{N+25-21}+\Rh_{\bar N}\p{N+25-22}\\
&=\Rh_{\bar N}\p{-N + 14}+\Rh_{\bar N}\p{N + 4}+\Rh_{\bar N}\p{N + 3}\\
&=0+\p{N + 3}+\p{N + 2}=\mbb{2N + 5}\\
&(N\geq14)
\end{align*}
\begin{align*}
\mbb{\Rh_{\bar N}\p{N+26}}&=\Rh_{\bar N}\p{N+26-\Rh_{\bar N}\p{N+25}}+\Rh_{\bar N}\p{N+26-\Rh_{\bar N}\p{N+24}}\\
&\phantom{=}+\Rh_{\bar N}\p{N+26-\Rh_{\bar N}\p{N+23}}\\
&=\Rh_{\bar N}\p{N+26-\p{2N + 5}}+\Rh_{\bar N}\p{N+26-\p{2N + 11}}+\Rh_{\bar N}\p{N+26-21}\\
&=\Rh_{\bar N}\p{-N + 21}+\Rh_{\bar N}\p{-N + 15}+\Rh_{\bar N}\p{N + 5}=0+0+9=\mbb{9}\\
&(N\geq21)
\end{align*}
\begin{align*}
\mbb{\Rh_{\bar N}\p{N+27}}&=\Rh_{\bar N}\p{N+27-\Rh_{\bar N}\p{N+26}}+\Rh_{\bar N}\p{N+27-\Rh_{\bar N}\p{N+25}}\\
&\phantom{=}+\Rh_{\bar N}\p{N+27-\Rh_{\bar N}\p{N+24}}\\
&=\Rh_{\bar N}\p{N+27-9}+\Rh_{\bar N}\p{N+27-\p{2N + 5}}+\Rh_{\bar N}\p{N+27-\p{2N + 11}}\\
&=\Rh_{\bar N}\p{N + 18}+\Rh_{\bar N}\p{-N + 22}+\Rh_{\bar N}\p{-N + 16}=18+0+0=\mbb{18}\\
&(N\geq22)
\end{align*}
\begin{align*}
\mbb{\Rh_{\bar N}\p{N+28}}&=\Rh_{\bar N}\p{N+28-\Rh_{\bar N}\p{N+27}}+\Rh_{\bar N}\p{N+28-\Rh_{\bar N}\p{N+26}}\\
&\phantom{=}+\Rh_{\bar N}\p{N+28-\Rh_{\bar N}\p{N+25}}\\
&=\Rh_{\bar N}\p{N+28-18}+\Rh_{\bar N}\p{N+28-9}+\Rh_{\bar N}\p{N+28-\p{2N + 5}}\\
&=\Rh_{\bar N}\p{N + 10}+\Rh_{\bar N}\p{N + 19}+\Rh_{\bar N}\p{-N + 23}\\
&=\p{N + 7}+\p{N + 13}+0=\mbb{2N + 20}\\
&(N\geq23)
\end{align*}
\begin{align*}
\mbb{\Rh_{\bar N}\p{N+29}}&=\Rh_{\bar N}\p{N+29-\Rh_{\bar N}\p{N+28}}+\Rh_{\bar N}\p{N+29-\Rh_{\bar N}\p{N+27}}\\
&\phantom{=}+\Rh_{\bar N}\p{N+29-\Rh_{\bar N}\p{N+26}}\\
&=\Rh_{\bar N}\p{N+29-\p{2N + 20}}+\Rh_{\bar N}\p{N+29-18}+\Rh_{\bar N}\p{N+29-9}\\
&=\Rh_{\bar N}\p{-N + 9}+\Rh_{\bar N}\p{N + 11}+\Rh_{\bar N}\p{N + 20}\\
&=0+\p{N + 8}+\p{N + 15}=\mbb{2N + 23}\\
&(N\geq9)
\end{align*}
\begin{align*}
\mbb{\Rh_{\bar N}\p{N+30}}&=\Rh_{\bar N}\p{N+30-\Rh_{\bar N}\p{N+29}}+\Rh_{\bar N}\p{N+30-\Rh_{\bar N}\p{N+28}}\\
&\phantom{=}+\Rh_{\bar N}\p{N+30-\Rh_{\bar N}\p{N+27}}\\
&=\Rh_{\bar N}\p{N+30-\p{2N + 23}}+\Rh_{\bar N}\p{N+30-\p{2N + 20}}+\Rh_{\bar N}\p{N+30-18}\\
&=\Rh_{\bar N}\p{-N + 7}+\Rh_{\bar N}\p{-N + 10}+\Rh_{\bar N}\p{N + 12}=0+0+\p{N + 9}=\mbb{N + 9}\\
&(N\geq10)
\end{align*}
\begin{align*}
\mbb{\Rh_{\bar N}\p{N+31}}&=\Rh_{\bar N}\p{N+31-\Rh_{\bar N}\p{N+30}}+\Rh_{\bar N}\p{N+31-\Rh_{\bar N}\p{N+29}}\\
&\phantom{=}+\Rh_{\bar N}\p{N+31-\Rh_{\bar N}\p{N+28}}\\
&=\Rh_{\bar N}\p{N+31-\p{N + 9}}+\Rh_{\bar N}\p{N+31-\p{2N + 23}}\\
&\phantom{=}+\Rh_{\bar N}\p{N+31-\p{2N + 20}}\\
&=\Rh_{\bar N}\p{22}+\Rh_{\bar N}\p{-N + 8}+\Rh_{\bar N}\p{-N + 11}=22+0+0=\mbb{22}\\
&(N\geq22)
\end{align*}
\begin{align*}
\mbb{\Rh_{\bar N}\p{N+32}}&=\Rh_{\bar N}\p{N+32-\Rh_{\bar N}\p{N+31}}+\Rh_{\bar N}\p{N+32-\Rh_{\bar N}\p{N+30}}\\
&\phantom{=}+\Rh_{\bar N}\p{N+32-\Rh_{\bar N}\p{N+29}}\\
&=\Rh_{\bar N}\p{N+32-22}+\Rh_{\bar N}\p{N+32-\p{N + 9}}+\Rh_{\bar N}\p{N+32-\p{2N + 23}}\\
&=\Rh_{\bar N}\p{N + 10}+\Rh_{\bar N}\p{23}+\Rh_{\bar N}\p{-N + 9}=\p{N + 7}+23+0=\mbb{N + 30}\\
&(N\geq23)
\end{align*}
\begin{align*}
\mbb{\Rh_{\bar N}\p{N+33}}&=\Rh_{\bar N}\p{N+33-\Rh_{\bar N}\p{N+32}}+\Rh_{\bar N}\p{N+33-\Rh_{\bar N}\p{N+31}}\\
&\phantom{=}+\Rh_{\bar N}\p{N+33-\Rh_{\bar N}\p{N+30}}\\
&=\Rh_{\bar N}\p{N+33-\p{N + 30}}+\Rh_{\bar N}\p{N+33-22}+\Rh_{\bar N}\p{N+33-\p{N + 9}}\\
&=\Rh_{\bar N}\p{3}+\Rh_{\bar N}\p{N + 11}+\Rh_{\bar N}\p{24}=3+\p{N + 8}+24=\mbb{N + 35}\\
&(N\geq24)
\end{align*}
\begin{align*}
\mbb{\Rh_{\bar N}\p{N+34}}&=\Rh_{\bar N}\p{N+34-\Rh_{\bar N}\p{N+33}}+\Rh_{\bar N}\p{N+34-\Rh_{\bar N}\p{N+32}}\\
&\phantom{=}+\Rh_{\bar N}\p{N+34-\Rh_{\bar N}\p{N+31}}\\
&=\Rh_{\bar N}\p{N+34-\p{N + 35}}+\Rh_{\bar N}\p{N+34-\p{N + 30}}+\Rh_{\bar N}\p{N+34-22}\\
&=\Rh_{\bar N}\p{-1}+\Rh_{\bar N}\p{4}+\Rh_{\bar N}\p{N + 12}=0+4+\p{N + 9}=\mbb{N + 13}\\
&(N\geq4)
\end{align*}
\begin{align*}
\mbb{\Rh_{\bar N}\p{N+35}}&=\Rh_{\bar N}\p{N+35-\Rh_{\bar N}\p{N+34}}+\Rh_{\bar N}\p{N+35-\Rh_{\bar N}\p{N+33}}\\
&\phantom{=}+\Rh_{\bar N}\p{N+35-\Rh_{\bar N}\p{N+32}}\\
&=\Rh_{\bar N}\p{N+35-\p{N + 13}}+\Rh_{\bar N}\p{N+35-\p{N + 35}}\\
&\phantom{=}+\Rh_{\bar N}\p{N+35-\p{N + 30}}\\
&=\Rh_{\bar N}\p{22}+\Rh_{\bar N}\p{0}+\Rh_{\bar N}\p{5}=22+0+5=\mbb{27}\\
&(N\geq22)
\end{align*}
\begin{align*}
\mbb{\Rh_{\bar N}\p{N+36}}&=\Rh_{\bar N}\p{N+36-\Rh_{\bar N}\p{N+35}}+\Rh_{\bar N}\p{N+36-\Rh_{\bar N}\p{N+34}}\\
&\phantom{=}+\Rh_{\bar N}\p{N+36-\Rh_{\bar N}\p{N+33}}\\
&=\Rh_{\bar N}\p{N+36-27}+\Rh_{\bar N}\p{N+36-\p{N + 13}}+\Rh_{\bar N}\p{N+36-\p{N + 35}}\\
&=\Rh_{\bar N}\p{N + 9}+\Rh_{\bar N}\p{23}+\Rh_{\bar N}\p{1}=12+23+1=\mbb{36}\\
&(N\geq23)
\end{align*}
\begin{align*}
\mbb{\Rh_{\bar N}\p{N+37}}&=\Rh_{\bar N}\p{N+37-\Rh_{\bar N}\p{N+36}}+\Rh_{\bar N}\p{N+37-\Rh_{\bar N}\p{N+35}}\\
&\phantom{=}+\Rh_{\bar N}\p{N+37-\Rh_{\bar N}\p{N+34}}\\
&=\Rh_{\bar N}\p{N+37-36}+\Rh_{\bar N}\p{N+37-27}+\Rh_{\bar N}\p{N+37-\p{N + 13}}\\
&=\Rh_{\bar N}\p{N + 1}+\Rh_{\bar N}\p{N + 10}+\Rh_{\bar N}\p{24}=6+\p{N + 7}+24=\mbb{N + 37}\\
&(N\geq24)
\end{align*}
\begin{align*}
\mbb{\Rh_{\bar N}\p{N+38}}&=\Rh_{\bar N}\p{N+38-\Rh_{\bar N}\p{N+37}}+\Rh_{\bar N}\p{N+38-\Rh_{\bar N}\p{N+36}}\\
&\phantom{=}+\Rh_{\bar N}\p{N+38-\Rh_{\bar N}\p{N+35}}\\
&=\Rh_{\bar N}\p{N+38-\p{N + 37}}+\Rh_{\bar N}\p{N+38-36}+\Rh_{\bar N}\p{N+38-27}\\
&=\Rh_{\bar N}\p{1}+\Rh_{\bar N}\p{N + 2}+\Rh_{\bar N}\p{N + 11}=1+\p{N + 1}+\p{N + 8}=\mbb{2N + 10}\\
&(N\geq1)
\end{align*}
\begin{align*}
\mbb{\Rh_{\bar N}\p{N+39}}&=\Rh_{\bar N}\p{N+39-\Rh_{\bar N}\p{N+38}}+\Rh_{\bar N}\p{N+39-\Rh_{\bar N}\p{N+37}}\\
&\phantom{=}+\Rh_{\bar N}\p{N+39-\Rh_{\bar N}\p{N+36}}\\
&=\Rh_{\bar N}\p{N+39-\p{2N + 10}}+\Rh_{\bar N}\p{N+39-\p{N + 37}}+\Rh_{\bar N}\p{N+39-36}\\
&=\Rh_{\bar N}\p{-N + 29}+\Rh_{\bar N}\p{2}+\Rh_{\bar N}\p{N + 3}=0+2+\p{N + 2}=\mbb{N + 4}\\
&(N\geq29)
\end{align*}
\begin{align*}
\mbb{\Rh_{\bar N}\p{N+40}}&=\Rh_{\bar N}\p{N+40-\Rh_{\bar N}\p{N+39}}+\Rh_{\bar N}\p{N+40-\Rh_{\bar N}\p{N+38}}\\
&\phantom{=}+\Rh_{\bar N}\p{N+40-\Rh_{\bar N}\p{N+37}}\\
&=\Rh_{\bar N}\p{N+40-\p{N + 4}}+\Rh_{\bar N}\p{N+40-\p{2N + 10}}\\
&\phantom{=}+\Rh_{\bar N}\p{N+40-\p{N + 37}}\\
&=\Rh_{\bar N}\p{36}+\Rh_{\bar N}\p{-N + 30}+\Rh_{\bar N}\p{3}=36+0+3=\mbb{39}\\
&(N\geq36)
\end{align*}
\begin{align*}
\mbb{\Rh_{\bar N}\p{N+41}}&=\Rh_{\bar N}\p{N+41-\Rh_{\bar N}\p{N+40}}+\Rh_{\bar N}\p{N+41-\Rh_{\bar N}\p{N+39}}\\
&\phantom{=}+\Rh_{\bar N}\p{N+41-\Rh_{\bar N}\p{N+38}}\\
&=\Rh_{\bar N}\p{N+41-39}+\Rh_{\bar N}\p{N+41-\p{N + 4}}+\Rh_{\bar N}\p{N+41-\p{2N + 10}}\\
&=\Rh_{\bar N}\p{N + 2}+\Rh_{\bar N}\p{37}+\Rh_{\bar N}\p{-N + 31}=\p{N + 1}+37+0=\mbb{N + 38}\\
&(N\geq37)
\end{align*}
\begin{align*}
\mbb{\Rh_{\bar N}\p{N+42}}&=\Rh_{\bar N}\p{N+42-\Rh_{\bar N}\p{N+41}}+\Rh_{\bar N}\p{N+42-\Rh_{\bar N}\p{N+40}}\\
&\phantom{=}+\Rh_{\bar N}\p{N+42-\Rh_{\bar N}\p{N+39}}\\
&=\Rh_{\bar N}\p{N+42-\p{N + 38}}+\Rh_{\bar N}\p{N+42-39}+\Rh_{\bar N}\p{N+42-\p{N + 4}}\\
&=\Rh_{\bar N}\p{4}+\Rh_{\bar N}\p{N + 3}+\Rh_{\bar N}\p{38}=4+\p{N + 2}+38=\mbb{N + 44}\\
&(N\geq38)
\end{align*}
\begin{align*}
\mbb{\Rh_{\bar N}\p{N+43}}&=\Rh_{\bar N}\p{N+43-\Rh_{\bar N}\p{N+42}}+\Rh_{\bar N}\p{N+43-\Rh_{\bar N}\p{N+41}}\\
&\phantom{=}+\Rh_{\bar N}\p{N+43-\Rh_{\bar N}\p{N+40}}\\
&=\Rh_{\bar N}\p{N+43-\p{N + 44}}+\Rh_{\bar N}\p{N+43-\p{N + 38}}+\Rh_{\bar N}\p{N+43-39}\\
&=\Rh_{\bar N}\p{-1}+\Rh_{\bar N}\p{5}+\Rh_{\bar N}\p{N + 4}=0+5+\p{N + 3}=\mbb{N + 8}\\
&(N\geq5)
\end{align*}
\begin{align*}
\mbb{\Rh_{\bar N}\p{N+44}}&=\Rh_{\bar N}\p{N+44-\Rh_{\bar N}\p{N+43}}+\Rh_{\bar N}\p{N+44-\Rh_{\bar N}\p{N+42}}\\
&\phantom{=}+\Rh_{\bar N}\p{N+44-\Rh_{\bar N}\p{N+41}}\\
&=\Rh_{\bar N}\p{N+44-\p{N + 8}}+\Rh_{\bar N}\p{N+44-\p{N + 44}}\\
&\phantom{=}+\Rh_{\bar N}\p{N+44-\p{N + 38}}\\
&=\Rh_{\bar N}\p{36}+\Rh_{\bar N}\p{0}+\Rh_{\bar N}\p{6}=36+0+6=\mbb{42}\\
&(N\geq36)
\end{align*}
\begin{align*}
\mbb{\Rh_{\bar N}\p{N+45}}&=\Rh_{\bar N}\p{N+45-\Rh_{\bar N}\p{N+44}}+\Rh_{\bar N}\p{N+45-\Rh_{\bar N}\p{N+43}}\\
&\phantom{=}+\Rh_{\bar N}\p{N+45-\Rh_{\bar N}\p{N+42}}\\
&=\Rh_{\bar N}\p{N+45-42}+\Rh_{\bar N}\p{N+45-\p{N + 8}}+\Rh_{\bar N}\p{N+45-\p{N + 44}}\\
&=\Rh_{\bar N}\p{N + 3}+\Rh_{\bar N}\p{37}+\Rh_{\bar N}\p{1}=\p{N + 2}+37+1=\mbb{N + 40}\\
&(N\geq37)
\end{align*}
\begin{align*}
\mbb{\Rh_{\bar N}\p{N+46}}&=\Rh_{\bar N}\p{N+46-\Rh_{\bar N}\p{N+45}}+\Rh_{\bar N}\p{N+46-\Rh_{\bar N}\p{N+44}}\\
&\phantom{=}+\Rh_{\bar N}\p{N+46-\Rh_{\bar N}\p{N+43}}\\
&=\Rh_{\bar N}\p{N+46-\p{N + 40}}+\Rh_{\bar N}\p{N+46-42}+\Rh_{\bar N}\p{N+46-\p{N + 8}}\\
&=\Rh_{\bar N}\p{6}+\Rh_{\bar N}\p{N + 4}+\Rh_{\bar N}\p{38}=6+\p{N + 3}+38=\mbb{N + 47}\\
&(N\geq38)
\end{align*}
\begin{align*}
\mbb{\Rh_{\bar N}\p{N+47}}&=\Rh_{\bar N}\p{N+47-\Rh_{\bar N}\p{N+46}}+\Rh_{\bar N}\p{N+47-\Rh_{\bar N}\p{N+45}}\\
&\phantom{=}+\Rh_{\bar N}\p{N+47-\Rh_{\bar N}\p{N+44}}\\
&=\Rh_{\bar N}\p{N+47-\p{N + 47}}+\Rh_{\bar N}\p{N+47-\p{N + 40}}\\
&\phantom{=}+\Rh_{\bar N}\p{N+47-42}\\
&=\Rh_{\bar N}\p{0}+\Rh_{\bar N}\p{7}+\Rh_{\bar N}\p{N + 5}=0+7+9=\mbb{16}\\
&(N\geq7)
\end{align*}
\begin{align*}
\mbb{\Rh_{\bar N}\p{N+48}}&=\Rh_{\bar N}\p{N+48-\Rh_{\bar N}\p{N+47}}+\Rh_{\bar N}\p{N+48-\Rh_{\bar N}\p{N+46}}\\
&\phantom{=}+\Rh_{\bar N}\p{N+48-\Rh_{\bar N}\p{N+45}}\\
&=\Rh_{\bar N}\p{N+48-16}+\Rh_{\bar N}\p{N+48-\p{N + 47}}+\Rh_{\bar N}\p{N+48-\p{N + 40}}\\
&=\Rh_{\bar N}\p{N + 32}+\Rh_{\bar N}\p{1}+\Rh_{\bar N}\p{8}=\p{N + 30}+1+8=\mbb{N + 39}\\
&(N\geq8)
\end{align*}
\begin{align*}
\mbb{\Rh_{\bar N}\p{N+49}}&=\Rh_{\bar N}\p{N+49-\Rh_{\bar N}\p{N+48}}+\Rh_{\bar N}\p{N+49-\Rh_{\bar N}\p{N+47}}\\
&\phantom{=}+\Rh_{\bar N}\p{N+49-\Rh_{\bar N}\p{N+46}}\\
&=\Rh_{\bar N}\p{N+49-\p{N + 39}}+\Rh_{\bar N}\p{N+49-16}+\Rh_{\bar N}\p{N+49-\p{N + 47}}\\
&=\Rh_{\bar N}\p{10}+\Rh_{\bar N}\p{N + 33}+\Rh_{\bar N}\p{2}=10+\p{N + 35}+2=\mbb{N + 47}\\
&(N\geq10)
\end{align*}
\begin{align*}
\mbb{\Rh_{\bar N}\p{N+50}}&=\Rh_{\bar N}\p{N+50-\Rh_{\bar N}\p{N+49}}+\Rh_{\bar N}\p{N+50-\Rh_{\bar N}\p{N+48}}\\
&\phantom{=}+\Rh_{\bar N}\p{N+50-\Rh_{\bar N}\p{N+47}}\\
&=\Rh_{\bar N}\p{N+50-\p{N + 47}}+\Rh_{\bar N}\p{N+50-\p{N + 39}}+\Rh_{\bar N}\p{N+50-16}\\
&=\Rh_{\bar N}\p{3}+\Rh_{\bar N}\p{11}+\Rh_{\bar N}\p{N + 34}=3+11+\p{N + 13}=\mbb{N + 27}\\
&(N\geq11)
\end{align*}
\begin{align*}
\mbb{\Rh_{\bar N}\p{N+51}}&=\Rh_{\bar N}\p{N+51-\Rh_{\bar N}\p{N+50}}+\Rh_{\bar N}\p{N+51-\Rh_{\bar N}\p{N+49}}\\
&\phantom{=}+\Rh_{\bar N}\p{N+51-\Rh_{\bar N}\p{N+48}}\\
&=\Rh_{\bar N}\p{N+51-\p{N + 27}}+\Rh_{\bar N}\p{N+51-\p{N + 47}}\\
&\phantom{=}+\Rh_{\bar N}\p{N+51-\p{N + 39}}\\
&=\Rh_{\bar N}\p{24}+\Rh_{\bar N}\p{4}+\Rh_{\bar N}\p{12}=24+4+12=\mbb{40}\\
&(N\geq24)
\end{align*}
\begin{align*}
\mbb{\Rh_{\bar N}\p{N+52}}&=\Rh_{\bar N}\p{N+52-\Rh_{\bar N}\p{N+51}}+\Rh_{\bar N}\p{N+52-\Rh_{\bar N}\p{N+50}}\\
&\phantom{=}+\Rh_{\bar N}\p{N+52-\Rh_{\bar N}\p{N+49}}\\
&=\Rh_{\bar N}\p{N+52-40}+\Rh_{\bar N}\p{N+52-\p{N + 27}}+\Rh_{\bar N}\p{N+52-\p{N + 47}}\\
&=\Rh_{\bar N}\p{N + 12}+\Rh_{\bar N}\p{25}+\Rh_{\bar N}\p{5}=\p{N + 9}+25+5=\mbb{N + 39}\\
&(N\geq25)
\end{align*}
\begin{align*}
\mbb{\Rh_{\bar N}\p{N+53}}&=\Rh_{\bar N}\p{N+53-\Rh_{\bar N}\p{N+52}}+\Rh_{\bar N}\p{N+53-\Rh_{\bar N}\p{N+51}}\\
&\phantom{=}+\Rh_{\bar N}\p{N+53-\Rh_{\bar N}\p{N+50}}\\
&=\Rh_{\bar N}\p{N+53-\p{N + 39}}+\Rh_{\bar N}\p{N+53-40}+\Rh_{\bar N}\p{N+53-\p{N + 27}}\\
&=\Rh_{\bar N}\p{14}+\Rh_{\bar N}\p{N + 13}+\Rh_{\bar N}\p{26}=14+15+26=\mbb{55}\\
&(N\geq26)
\end{align*}
\begin{align*}
\mbb{\Rh_{\bar N}\p{N+54}}&=\Rh_{\bar N}\p{N+54-\Rh_{\bar N}\p{N+53}}+\Rh_{\bar N}\p{N+54-\Rh_{\bar N}\p{N+52}}\\
&\phantom{=}+\Rh_{\bar N}\p{N+54-\Rh_{\bar N}\p{N+51}}\\
&=\Rh_{\bar N}\p{N+54-55}+\Rh_{\bar N}\p{N+54-\p{N + 39}}+\Rh_{\bar N}\p{N+54-40}\\
&=\Rh_{\bar N}\p{N - 1}+\Rh_{\bar N}\p{15}+\Rh_{\bar N}\p{N + 14}\\
&=\p{N - 1}+15+\p{N + 10}=\mbb{2N + 24}\\
&(N\geq15)
\end{align*}
\begin{align*}
\mbb{\Rh_{\bar N}\p{N+55}}&=\Rh_{\bar N}\p{N+55-\Rh_{\bar N}\p{N+54}}+\Rh_{\bar N}\p{N+55-\Rh_{\bar N}\p{N+53}}\\
&\phantom{=}+\Rh_{\bar N}\p{N+55-\Rh_{\bar N}\p{N+52}}\\
&=\Rh_{\bar N}\p{N+55-\p{2N + 24}}+\Rh_{\bar N}\p{N+55-55}+\Rh_{\bar N}\p{N+55-\p{N + 39}}\\
&=\Rh_{\bar N}\p{-N + 31}+\Rh_{\bar N}\p{N}+\Rh_{\bar N}\p{16}=0+N+16=\mbb{N + 16}\\
&(N\geq31)
\end{align*}
\begin{align*}
\mbb{\Rh_{\bar N}\p{N+56}}&=\Rh_{\bar N}\p{N+56-\Rh_{\bar N}\p{N+55}}+\Rh_{\bar N}\p{N+56-\Rh_{\bar N}\p{N+54}}\\
&\phantom{=}+\Rh_{\bar N}\p{N+56-\Rh_{\bar N}\p{N+53}}\\
&=\Rh_{\bar N}\p{N+56-\p{N + 16}}+\Rh_{\bar N}\p{N+56-\p{2N + 24}}+\Rh_{\bar N}\p{N+56-55}\\
&=\Rh_{\bar N}\p{40}+\Rh_{\bar N}\p{-N + 32}+\Rh_{\bar N}\p{N + 1}=40+0+6=\mbb{46}\\
&(N\geq40)
\end{align*}
\begin{align*}
\mbb{\Rh_{\bar N}\p{N+57}}&=\Rh_{\bar N}\p{N+57-\Rh_{\bar N}\p{N+56}}+\Rh_{\bar N}\p{N+57-\Rh_{\bar N}\p{N+55}}\\
&\phantom{=}+\Rh_{\bar N}\p{N+57-\Rh_{\bar N}\p{N+54}}\\
&=\Rh_{\bar N}\p{N+57-46}+\Rh_{\bar N}\p{N+57-\p{N + 16}}+\Rh_{\bar N}\p{N+57-\p{2N + 24}}\\
&=\Rh_{\bar N}\p{N + 11}+\Rh_{\bar N}\p{41}+\Rh_{\bar N}\p{-N + 33}=\p{N + 8}+41+0=\mbb{N + 49}\\
&(N\geq41)
\end{align*}
\begin{align*}
\mbb{\Rh_{\bar N}\p{N+58}}&=\Rh_{\bar N}\p{N+58-\Rh_{\bar N}\p{N+57}}+\Rh_{\bar N}\p{N+58-\Rh_{\bar N}\p{N+56}}\\
&\phantom{=}+\Rh_{\bar N}\p{N+58-\Rh_{\bar N}\p{N+55}}\\
&=\Rh_{\bar N}\p{N+58-\p{N + 49}}+\Rh_{\bar N}\p{N+58-46}+\Rh_{\bar N}\p{N+58-\p{N + 16}}\\
&=\Rh_{\bar N}\p{9}+\Rh_{\bar N}\p{N + 12}+\Rh_{\bar N}\p{42}=9+\p{N + 9}+42=\mbb{N + 60}\\
&(N\geq42)
\end{align*}
\begin{align*}
\mbb{\Rh_{\bar N}\p{N+59}}&=\Rh_{\bar N}\p{N+59-\Rh_{\bar N}\p{N+58}}+\Rh_{\bar N}\p{N+59-\Rh_{\bar N}\p{N+57}}\\
&\phantom{=}+\Rh_{\bar N}\p{N+59-\Rh_{\bar N}\p{N+56}}\\
&=\Rh_{\bar N}\p{N+59-\p{N + 60}}+\Rh_{\bar N}\p{N+59-\p{N + 49}}+\Rh_{\bar N}\p{N+59-46}\\
&=\Rh_{\bar N}\p{-1}+\Rh_{\bar N}\p{10}+\Rh_{\bar N}\p{N + 13}=0+10+15=\mbb{25}\\
&(N\geq10)
\end{align*}
\begin{align*}
\mbb{\Rh_{\bar N}\p{N+60}}&=\Rh_{\bar N}\p{N+60-\Rh_{\bar N}\p{N+59}}+\Rh_{\bar N}\p{N+60-\Rh_{\bar N}\p{N+58}}\\
&\phantom{=}+\Rh_{\bar N}\p{N+60-\Rh_{\bar N}\p{N+57}}\\
&=\Rh_{\bar N}\p{N+60-25}+\Rh_{\bar N}\p{N+60-\p{N + 60}}+\Rh_{\bar N}\p{N+60-\p{N + 49}}\\
&=\Rh_{\bar N}\p{N + 35}+\Rh_{\bar N}\p{0}+\Rh_{\bar N}\p{11}=27+0+11=\mbb{38}\\
&(N\geq11)
\end{align*}
\begin{align*}
\mbb{\Rh_{\bar N}\p{N+61}}&=\Rh_{\bar N}\p{N+61-\Rh_{\bar N}\p{N+60}}+\Rh_{\bar N}\p{N+61-\Rh_{\bar N}\p{N+59}}\\
&\phantom{=}+\Rh_{\bar N}\p{N+61-\Rh_{\bar N}\p{N+58}}\\
&=\Rh_{\bar N}\p{N+61-38}+\Rh_{\bar N}\p{N+61-25}+\Rh_{\bar N}\p{N+61-\p{N + 60}}\\
&=\Rh_{\bar N}\p{N + 23}+\Rh_{\bar N}\p{N + 36}+\Rh_{\bar N}\p{1}=21+36+1=\mbb{58}\\
&(N\geq1)
\end{align*}
\begin{align*}
\mbb{\Rh_{\bar N}\p{N+62}}&=\Rh_{\bar N}\p{N+62-\Rh_{\bar N}\p{N+61}}+\Rh_{\bar N}\p{N+62-\Rh_{\bar N}\p{N+60}}\\
&\phantom{=}+\Rh_{\bar N}\p{N+62-\Rh_{\bar N}\p{N+59}}\\
&=\Rh_{\bar N}\p{N+62-58}+\Rh_{\bar N}\p{N+62-38}+\Rh_{\bar N}\p{N+62-25}\\
&=\Rh_{\bar N}\p{N + 4}+\Rh_{\bar N}\p{N + 24}+\Rh_{\bar N}\p{N + 37}\\
&=\p{N + 3}+\p{2N + 11}+\p{N + 37}=\mbb{4N + 51}\\
&(N\geq1)
\end{align*}
\begin{align*}
\mbb{\Rh_{\bar N}\p{N+63}}&=\Rh_{\bar N}\p{N+63-\Rh_{\bar N}\p{N+62}}+\Rh_{\bar N}\p{N+63-\Rh_{\bar N}\p{N+61}}\\
&\phantom{=}+\Rh_{\bar N}\p{N+63-\Rh_{\bar N}\p{N+60}}\\
&=\Rh_{\bar N}\p{N+63-\p{4N + 51}}+\Rh_{\bar N}\p{N+63-58}+\Rh_{\bar N}\p{N+63-38}\\
&=\Rh_{\bar N}\p{-3N + 12}+\Rh_{\bar N}\p{N + 5}+\Rh_{\bar N}\p{N + 25}\\
&=0+9+\p{2N + 5}=\mbb{2N + 14}\\
&(N\geq4)
\end{align*}
\begin{align*}
\mbb{\Rh_{\bar N}\p{N+64}}&=\Rh_{\bar N}\p{N+64-\Rh_{\bar N}\p{N+63}}+\Rh_{\bar N}\p{N+64-\Rh_{\bar N}\p{N+62}}\\
&\phantom{=}+\Rh_{\bar N}\p{N+64-\Rh_{\bar N}\p{N+61}}\\
&=\Rh_{\bar N}\p{N+64-\p{2N + 14}}+\Rh_{\bar N}\p{N+64-\p{4N + 51}}+\Rh_{\bar N}\p{N+64-58}\\
&=\Rh_{\bar N}\p{-N + 50}+\Rh_{\bar N}\p{-3N + 13}+\Rh_{\bar N}\p{N + 6}\\
&=0+0+\p{N + 4}=\mbb{N + 4}\\
&(N\geq50)
\end{align*}
\begin{align*}
\mbb{\Rh_{\bar N}\p{N+65}}&=\Rh_{\bar N}\p{N+65-\Rh_{\bar N}\p{N+64}}+\Rh_{\bar N}\p{N+65-\Rh_{\bar N}\p{N+63}}\\
&\phantom{=}+\Rh_{\bar N}\p{N+65-\Rh_{\bar N}\p{N+62}}\\
&=\Rh_{\bar N}\p{N+65-\p{N + 4}}+\Rh_{\bar N}\p{N+65-\p{2N + 14}}\\
&\phantom{=}+\Rh_{\bar N}\p{N+65-\p{4N + 51}}\\
&=\Rh_{\bar N}\p{61}+\Rh_{\bar N}\p{-N + 51}+\Rh_{\bar N}\p{-3N + 14}=61+0+0=\mbb{61}\\
&(N\geq61)
\end{align*}
\begin{align*}
\mbb{\Rh_{\bar N}\p{N+66}}&=\Rh_{\bar N}\p{N+66-\Rh_{\bar N}\p{N+65}}+\Rh_{\bar N}\p{N+66-\Rh_{\bar N}\p{N+64}}\\
&\phantom{=}+\Rh_{\bar N}\p{N+66-\Rh_{\bar N}\p{N+63}}\\
&=\Rh_{\bar N}\p{N+66-61}+\Rh_{\bar N}\p{N+66-\p{N + 4}}+\Rh_{\bar N}\p{N+66-\p{2N + 14}}\\
&=\Rh_{\bar N}\p{N + 5}+\Rh_{\bar N}\p{62}+\Rh_{\bar N}\p{-N + 52}=9+62+0=\mbb{71}\\
&(N\geq62)
\end{align*}
\begin{align*}
\mbb{\Rh_{\bar N}\p{N+67}}&=\Rh_{\bar N}\p{N+67-\Rh_{\bar N}\p{N+66}}+\Rh_{\bar N}\p{N+67-\Rh_{\bar N}\p{N+65}}\\
&\phantom{=}+\Rh_{\bar N}\p{N+67-\Rh_{\bar N}\p{N+64}}\\
&=\Rh_{\bar N}\p{N+67-71}+\Rh_{\bar N}\p{N+67-61}+\Rh_{\bar N}\p{N+67-\p{N + 4}}\\
&=\Rh_{\bar N}\p{N - 4}+\Rh_{\bar N}\p{N + 6}+\Rh_{\bar N}\p{63}\\
&=\p{N - 4}+\p{N + 4}+63=\mbb{2N + 63}\\
&(N\geq63)
\end{align*}
\begin{align*}
\mbb{\Rh_{\bar N}\p{N+68}}&=\Rh_{\bar N}\p{N+68-\Rh_{\bar N}\p{N+67}}+\Rh_{\bar N}\p{N+68-\Rh_{\bar N}\p{N+66}}\\
&\phantom{=}+\Rh_{\bar N}\p{N+68-\Rh_{\bar N}\p{N+65}}\\
&=\Rh_{\bar N}\p{N+68-\p{2N + 63}}+\Rh_{\bar N}\p{N+68-71}+\Rh_{\bar N}\p{N+68-61}\\
&=\Rh_{\bar N}\p{-N + 5}+\Rh_{\bar N}\p{N - 3}+\Rh_{\bar N}\p{N + 7}\\
&=0+\p{N - 3}+\p{N + 5}=\mbb{2N + 2}\\
&(N\geq5)
\end{align*}
\begin{align*}
\mbb{\Rh_{\bar N}\p{N+69}}&=\Rh_{\bar N}\p{N+69-\Rh_{\bar N}\p{N+68}}+\Rh_{\bar N}\p{N+69-\Rh_{\bar N}\p{N+67}}\\
&\phantom{=}+\Rh_{\bar N}\p{N+69-\Rh_{\bar N}\p{N+66}}\\
&=\Rh_{\bar N}\p{N+69-\p{2N + 2}}+\Rh_{\bar N}\p{N+69-\p{2N + 63}}+\Rh_{\bar N}\p{N+69-71}\\
&=\Rh_{\bar N}\p{-N + 67}+\Rh_{\bar N}\p{-N + 6}+\Rh_{\bar N}\p{N - 2}=0+0+\p{N - 2}=\mbb{N - 2}\\
&(N\geq67)
\end{align*}


\end{document}